\documentclass[11pt,twoside,reqno]{amsart}
\usepackage[all]{xy}   
\usepackage {amssymb,latexsym,amsthm,amsmath,setspace,bm}
\usepackage[colorlinks, linkcolor=blue, citecolor=red]{hyperref}
\usepackage{graphicx}
\usepackage{mathtools}
\usepackage{xcolor}
\usepackage[utf8]{inputenc}
\usepackage{tikz}
\usepackage{enumitem}
\usepackage{newtxmath} 

\newcommand{\cross}{\times}

\newcommand{\secref}[1]{Section~\ref{#1}}
\newcommand{\thmref}[1]{Theorem~\ref{#1}}
\newcommand{\thmrefb}[1]{{\bf{Theorem}}~\ref{#1}}
\newcommand{\lemref}[1]{Lemma~\ref{#1}}
\newcommand{\remref}[1]{Remark~\ref{#1}}

\newcommand{\propref}[1]{Proposition~\ref{#1}}
\newcommand{\corref}[1]{Corollary~\ref{#1}}

\makeatletter
\def\imod#1{\allowbreak\mkern10mu({\operator@font mod}\,\,#1)}
\makeatother

\newtheorem{theorem}{Theorem}[subsection]
\newtheorem{lemma}[theorem]{Lemma}
\newtheorem{corollary}[theorem]{Corollary}

\newtheorem{proposition}[theorem]{Proposition}
\newtheorem*{theorem*}{Theorem}

\theoremstyle{definition}
\newtheorem{remark}[theorem]{Remark}

\newtheorem{definition}[theorem]{Definition}

\numberwithin{equation}{section}

\newcommand{\ignore}[1]{}

\newcommand{\nc}{\newcommand}

\nc{\Cal}{\cal} \nc{\Xp}[1]{X^+(#1)} \nc{\Xm}[1]{X^-(#1)}
\nc{\on}{\operatorname} \nc{\Z}{{\mathbb{Z}}} \nc{\J}{{\cal J}} \nc{\C}{{\mathbb{C}}} \nc{\Q}{{\bold Q}} \nc{\R}{{\mathbb{R}}}
 \nc{\K}{\bold{\kappa}}

\nc{\F}{\operatorname{\boF}}
\nc{\pr}{\operatorname{\boldsymbol{pr}}}
\nc{\ev}{\operatorname{{ev}}}
\nc{\ch}{\operatorname{{ch}}}
\nc{\hook}{\lceil} 
\nc{\HLn}{\noindent\makebox[\linewidth]{\rule{\textwidth}{1pt}}}
\nc\bxi{{\mbox{$\boldsymbol{\xi}$}}}
\nc\Blambda{{\mbox{$\boldsymbol{\lambda}$}}}

\nc{\N}{{\Bbb N}} \nc\boa{\bold a} \nc\bob{\bold b} \nc\boc{\bold c} \nc\bod{\bold d} \nc\boe{\bold e} \nc\bof{\bold f} \nc\bog{\bold g}
\nc\boh{\bold h} \nc\boi{\bold i} \nc\boj{\bold j} \nc\bok{\bold k} \nc
\bol{\bold l} \nc\bom{\bold m} \nc\bon{\bold n} \nc\boo{\bold o}
\nc\bop{\bold p} \nc\boq{\bold q} \nc\bor{\bold r} \nc\bos{\bold s} \nc\boT{\bold t} \nc\boF{\bold F} \nc\bou{\bold u} \nc\bov{\bold v}
\nc\bow{\bold w} \nc\boz{\bold z} \nc\boy{\bold y} \nc\ba{\bold A} \nc\bb{\bold B} \nc\bc{\bold C} \nc\bd{\bold D} \nc\be{\bold E} \nc\bg{\bold
G} \nc\bh{\bold H} \nc\bi{\bold I} \nc\bj{\bold J} \nc\bk{\bold K} \nc\bl{\bold L} \nc\bM{\bold M} \nc\bn{\bold N} \nc\bo{\bold O} \nc\bp{\bold
P} \nc\bq{\bold Q} \nc\br{\bold R} \nc\bs{\bold S} \nc\bt{\bold T} \nc\bu{\bold U} \nc\bv{\bold V} \nc\bw{\bold W} \nc\bz{\bold Z} \nc\bx{\bold
x} \nc\KR{\bold{KR}} \nc\rk{\bold{rk}} \nc\het{\text{ht }}

\nc\udi{\underline i} \nc\udj{\underline j}

\nc\fn{{fin}}  \nc\af{{aff}}  \nc\tr{{tor}} \nc\btilde{\bold{\tilde{\bold{H}}}}

\nc{\mpp}{\rotatebox[origin=c]{180}{\pm}}

\nc\eps{\epsilon}
\nc\toa{\tilde a} \nc\tob{\tilde b} \nc\toc{\tilde c} \nc\tod{\tilde d} \nc\toe{\tilde e} \nc\tof{\tilde f} \nc\tog{\tilde g} \nc\toh{\tilde h}
\nc\toi{\tilde i} \nc\toj{\tilde j} \nc\tok{\tilde k} \nc\tol{\tilde l} \nc\tom{\tilde m}  \nc\ton{\tilde n} \nc\too{\tilde o} \nc\toq{\tilde q}
\nc\tor{\tilde r} \nc\tos{\tilde s} \nc\toT{\tilde t} \nc\tou{\tilde u} \nc\tov{\tilde v} \nc\tow{\tilde w} \nc\toz{\tilde z}

\begin{document}
	\setcounter{section}{0}
	\setcounter{tocdepth}{1}
	
	\title[Filtration of tensor product of local Weyl modules by fusion modules]{Filtration of tensor product of local Weyl modules for $\mathfrak{sl}_{n+1}[t]$}
	\author[Divya Setia]{Divya Setia}
     \author{Shushma Rani}
	\author{Tanusree Khandai}
	\address{Indian Institute of Science Education and Research Mohali, Knowledge City,  Sector 81, S.A.S. Nagar 140306, Punjab, India}
	\email{divyasetia01@gmail.com}
	\email{tanusree@iisermohali.ac.in}
 \address{Harish Chandra Research Institute, A CI of Homi Bhaba National Institute, Chhatnag Road, Jhunsi, Allahabad, Uttar Pradesh, 211019, India,}
 \email{shushmarani@hri.res.in, shushmarani95@gmail.com}
\date{}
\thanks{}


%
%
%
%
%
%
%
\begin{abstract} 
In this paper, we consider the tensor product of local Weyl modules for $\mathfrak{sl}_{n+1}[t]$ whose highest weights are multiples of the first and $n^{th}$ fundamental weights. We determine the graded character of these tensor product modules in terms of the graded character of local Weyl modules and prove that these modules admit a filtration whose successive quotients are either truncated Weyl modules or fusion products of Demazure modules. Furthermore, we establish that the truncated Weyl modules appearing as quotients in the filtration of tensor products of local Weyl modules of $\mathfrak{sl}_3[t]$ are indeed isomorphic to fusion products of irreducible $\mathfrak{sl}_3[t]$-modules which establish the independence of a family of fusion product modules of $\mathfrak{sl}_3[t]$ from the set of its evaluation parameters.
\end{abstract}

 \maketitle
%
%
%
%
%
%
%

\section{INTRODUCTION}
Let $\mathfrak{g}$ be a finite-dimensional simple Lie algebra over the complex field $\mathbb{C}$, and let $\mathfrak{g}[t]$ be the associated current Lie algebra. In 2001, Chari and Pressley introduced a family of finite-dimensional graded modules for current Lie algebras known as the local Weyl modules (\cite{MR1850556}). Defined through generators and relations, these modules exhibit certain universal properties. It was shown in \cite{MR3371494} that the graded characters of local Weyl modules of current Lie algebras  coincide with non-symmetric Macdonald polynomials, specialized at $t = 0$. Using this connection between Macdonald polynomials and the representation theory of current Lie algebras of type $A$, we study the structure of tensor products of two local Weyl modules for $\mathfrak{sl}_{n+1}[t]$  whose highest weights are multiples of the first or $n^{th}$ fundamental weight.


Let $\mathfrak{\hat{g}}$ be an affine Kac-Moody Lie algebra. A Demazure module is a finite-dimensional representation of a Borel subalgebra of $\mathfrak{\hat{g}}$. It is generated by an extremal vector of a highest weight integrable $\mathfrak{\hat{g}}$-module, and its level is determined by the integer by which the central element of $\mathfrak{\hat{g}}$ acts on it. Demazure modules invariant under the action of the maximal parabolic subalgebra $\mathfrak{g}[t]$ of $\mathfrak{\hat{g}}$ are said to be $\mathfrak{g}[t]$-stable. 
It was demonstrated in \cite{MR1771615} for $\mathfrak{\hat{g}}$ of type $A_n^1$ and in \cite{MR1953294} for simply-laced untwisted affine Lie algebras and twisted affine Lie algebras, that the graded characters of $\mathfrak{g}[t]$-stable Demazure modules are given by specialized Macdonald polynomials. This played an important role in establishing that a local Weyl module for $\mathfrak{sl}_{n+1}[t]$ is indeed a level 1 Demazure module. Alternatively, using dimension arguments this result was proved in \cite{MR2323538} for current Lie algebras associated with simply-laced simple Lie algebras.

The notion of studying the structure of a $\mathfrak{\hat{g}}$-module using Demazure flags was introduced and developed in \cite{MR2214249} and the concept further explored in \cite{MR2855081} to understand the structure of local Weyl modules for current algebras associated with nonsimply-laced simple Lie algebras. This method was used in \cite{MR3210603} to obtain the graded characters of level 1 Demazure modules associated with the affine Kac-Moody Lie algebra of type $B_n^1$. Motivated by the research initiated in \cite{9} to examine the Demazure filtration of tensor products of stable Demazure modules for $\mathfrak{sl}_2[t]$,
 in \cite{24}, we investigated the structure of two distinct classes of tensor products of Demazure modules for $\mathfrak{sl}_{2}[t]$. This led us to prove results that support the conjecture in \cite{9} regarding the filtration of tensor products of Demazure modules of levels $m$ and $n$ by Demazure modules of level $m + n$. In  this paper, we explore  whether analogous results hold for tensor products of local Weyl modules  of $\mathfrak{sl}_{n+1}[t]$.

To understand the characters of level two Demazure modules, a family of finite-dimensional $\mathfrak{sl}_{n+1}[t]$-modules $M(\nu,\lambda)$, indexed by two dominant integral weights $\nu$ and $\lambda$, was introduced in \cite{23} and extensively studied in \cite{MR4224104}. These modules interpolate between level 1 and level 2 Demazure modules, and it has been demonstrated that their graded characters are symmetric polynomials in $n+1$ variables under suitable conditions on $(\nu,\lambda)$. Furthermore, it was established that for $n\leq 3$, $M(\nu,\lambda)$ has a filtration by level 2 Demazure modules. 
  Let $W_{\text(loc)}(\lambda)$ denote the local Weyl module of $\mathfrak{sl}_{n+1}[t]$ with highest weight $\lambda$. Denoting the $i^{th}$ fundamental weight of $\mathfrak{sl}_{n+1}$ by $\omega_i$, we establish that the $\mathfrak{sl}_{n+1}[t]$ modules of the form $W_{loc}(m\omega_i)\otimes W_{loc}(n\omega_i)$, for $i\in \{1,n\}$, possess a filtration by $M(\nu,\lambda)$-modules. This 
shows that for $n\leq 3$, such tensor product modules have a filtration by level 2 Demazure modules.

With the aim of constructing generalized versions of the Kostka polynomials, Feigin and Loktev introduced in \cite{MR1729359} a family of $\mathbb{Z}$-graded modules for current Lie algebras called fusion product modules.  The underlying vector space for these modules comprises tensor products of highest weight cyclic $\mathfrak{g}[t]$-modules, and their definition incorporates evaluation maps based on a finite subset of the complex field $\mathbb{C}$. It was conjectured that the structure of these modules remains unchanged regardless of the associated set of evaluation parameters. Over the past two decades, this conjecture has been confirmed in several instances showcasing that fusion product modules encompass various representations of current Lie algebras, including local Weyl Modules, $\mathfrak{g}[t]$-stable Demazure modules, generalized Demazure modules and the $M(\nu,\lambda)$-modules. These results have been proven across a range of papers, including \cite{MR3298990, MR3608149, MR3296163, MR2271991, MR1729359, MR2323538, MR3704743}.

 An important tool that has consistently helped in verifying the Feigin-Loktev conjecture is the use of Chari-Venkatesh modules, also known as CV-modules (\cite{MR3296163}).  To study of the $\mathfrak{sl}_3[t]$ modules $W_{loc}(m\omega_1)\otimes W_{loc}(n\omega_2)$, we introduce the family $M_j(\lambda_1,\lambda_2,\lambda_3)$-modules, which are indexed by an integer $j$ and a triple of dominant integral weights $(\lambda_1,\lambda_2,\lambda_3)$. For $j=0$, these modules interpolate between $M(\nu,\lambda)$-modules and a class of cyclic $\mathfrak{sl}_3[t]$ modules studied in \cite{MR3704743}. Denoting the fusion product of modules $U$ and $V$ based at $\{a_1,a_2\}$ by $U^{a_1}\ast V^{a_2}$,  the adjoint representation of $\mathfrak{sl}_3[t]$ by $V(\theta)$ and level $\ell$ Demazure module with highest weight $\ell\lambda_\ell$ by $D(\ell,\ell\lambda_\ell)$,   using properties of CV-modules we prove  that the $\mathfrak{sl}_3[t]$ module $M_j(\lambda_1,\lambda_2,\lambda_3)$ is isomorphic to the fusion product module 
$$D(3,3\lambda_3)^{a_1}\ast D(2,2\lambda_2)^{a_2}\ast D(1,\lambda_1)^{a_3}*
V(\theta)^{a_4}\ast \cdots \ast V(\theta)^{a_{j+3}},$$ when 
$\lambda_3=0$ or $\lambda_1\in \{\omega_i:i=1,2\}\cup \{0\}$. This proves the independence of this class of fusion module on the set of evaluation parameters $\{a_1,a_2,\cdots,a_{j+3}\}$. As an offshoot we are able to show that the truncated local Weyl module $W_{m_1+m_2+k}(m_1\omega_1+m_2\omega_2+k\theta)$ which is isomorphic to $M_k(m_1\omega_1+m_2\omega_2,0,0)$, 
is isomorphic to a fusion product of $k$-irreducible $\mathfrak{sl}_3[t]$-modules. This generalizes a similar result that was proved in \cite{MR3298990}.
This also helps us prove that the $\mathfrak{sl}_3[t]$ tensor product module  $W_{loc}(m\omega_1)\otimes W_{loc}(n\omega_2)$ admits a filtration by submodules whose successive quotients are isomorphic to truncated Weyl modules $W_{M-j}(m\omega_1+k\omega_2-j\theta)$, where $0\leq j\leq \min\{m,k\}$ and $M=\max\{m,k\}$.  

To establish the existence of filtrations for the tensor product modules of $\mathfrak{sl}_{n+1}[t]$, we first compute the graded characters of the modules using Pieri formulas. Subsequently, using an understanding of the Chari-Loktev basis of local Weyl modules and the graded character formulas, we show that the tensor product modules in question possess a filtrations whose successive quotients are either truncated Weyl modules or fusion product of Demazure modules. The parametrization of the Chari-Loktev basis by the partition overlaid patterns determined in \cite{pop}, facilitates the adoption of a systematic approach to study the problem.

This paper is organized as follows. In section $2$, we set the notations and recall the basic definitions that are required in the paper. In Section 3 we recall the results from \cite{MR2271991,pop} on the Chari-Loktev basis of local Weyl modules for $\mathfrak{sl}_{n+1}[t]$. 
In Section 4, we recall the definition and properties of CV modules from \cite{MR3296163} and determine suitable generators for the tensor products of local Weyl modules of $\mathfrak{sl}_{n+1}[t]$. Using Pieri formulas, we determine in section 5 the graded character formulas of the tensor products of two local Weyl modules and express the product of two specialized Macdonald polynomials in terms of specialized Macdonald polynomials.  We use these in sections 6 and 7 respectively to show that $\mathfrak{sl}_3[t]$-modules of the form
 $W_{loc}(m\omega_{1}) \otimes W_{loc}(l\omega_{2})$ have filtration by truncated Weyl modules and $\mathfrak{sl}_{n+1}[t]$ modules of the form  $W_{loc}(m\omega_{i}) \otimes W_{loc}(l\omega_{i})$, for $i\in \{1,n\}$ have filtrations by $M(\nu,\lambda)$-modules. Finally in section 8, we study the CV-modules $M_j(\lambda_1,\lambda_2,\lambda_3)$ which are indexed by special sets of triples of 
dominant integral weights $(\lambda_1,\lambda_2,\lambda_3)$ of $\mathfrak{sl}_3(\mathbb C)$ and $j\in \mathbb Z_+$ and complete the proof of the main theorems. 

%
%
%
%
%
%
%
%
\section{PRELIMINARIES}
Throughout this paper, $\mathbb C$ will denote the field of complex numbers and $\mathbb Z$ the set of integers, $\mathbb Z_+$ (resp. $\mathbb N$)  the set of non-negative integers (resp. positive integers). For $n,r \in \mathbb Z_+$, set 
$$ [n]_q= \frac{1-q^n}{1-q}, \quad \begin{bmatrix}
    n\\ r
\end{bmatrix}_q= \frac{[n]_q [n-1]_q \ldots [n-r+1]_q}{[r]_q[r-1]_q \ldots [1]_q}$$
$$\begin{bmatrix}
    n\\ 0
\end{bmatrix}_q=1, \quad \begin{bmatrix}
    n\\ 1
\end{bmatrix}_q= \frac{1-q^n}{1-q}, \quad \begin{bmatrix}
    n\\ r
\end{bmatrix}_q=0, \quad \text{unless } \{n,r,n-r\} \subset \mathbb Z_+.$$

\subsection{} Let $\mathfrak g= \mathfrak{sl}_{n+1}$ be the Lie algebra of $(n+1)\times(n+1)$ matrices with trace zero and $\mathfrak h$ be a Cartan subalgebra of $\mathfrak g$. Fix an indexing set $I= \{1, 2, \cdots, n\}$. Let $\Delta= \{\alpha_i: i \in I\}$ be a set of simple roots, $\theta= \sum\limits_{i \in I} \alpha_i$ be the highest root of $\mathfrak g $, $\{\omega_i: i \in I\} $ be the set of the fundamental weights and $P^+= \sum\limits_{i \in I} \mathbb Z_+ \omega_i$ ($Q^+= \sum\limits_{i \in I} \mathbb Z_+ \alpha_i)$ be the dominant weight lattice (resp. the positive root lattice).  Let $\Phi$ denote the set of roots of $\mathfrak g$ and $\Phi^+ = \Phi\cap Q^+$ be the set of positive roots. Fix Chevalley generators $\{x_\alpha, y_\alpha, h_\alpha: \alpha \in \Phi^+\}$ of $\mathfrak g$ such that $$ [h_\alpha, x_\alpha]= 2 x_\alpha, \quad  [h_\alpha, y_\alpha]= -2 y_\alpha, \quad  [x_\alpha, y_\alpha]=  h_\alpha. $$

With $\mathfrak n^+ = \bigoplus\limits_{\alpha\in \Phi^+} \mathbb C x_\alpha $ and $\mathfrak n^- = \bigoplus\limits_{\alpha\in \Phi^+} \mathbb C y_\alpha $, $\mathfrak g$ has the following triangular decomposition: $\mathfrak g = \mathfrak n^-\oplus \mathfrak h\oplus \mathfrak n^+$. For $\lambda \in P^+,$ let $V(\lambda)$ be the irreducible, cyclic $\mathfrak g$-module generated by $v_\lambda$ with defining relations:
\begin{equation}\label{irr.rel}
\mathfrak n^+.v_\lambda =0, \hspace{.25cm} h.v_\lambda=\lambda(h)v_\lambda, \hspace{.25cm} (y_{\alpha_i})^{\lambda(h_{\alpha_i})+1}v_\lambda =0, \quad \quad \forall \,\,  h\in \mathfrak h \, \text{and} \, i \in I. \end{equation}
Let $W$ be the Weyl group of $\mathfrak g$ and $w_0$ the longest element of $W$. Then for all $\lambda\in P^+$, the module $V(-w_0\lambda)$ is the $\mathfrak g$-dual of $V (\lambda)$. 
\\

Let $\mathbb {Z}[P]$ be the group ring of $P$ with integer coefficients and basis $\{e(\mu): \mu \in P\}$. The character of the $\mathfrak g$-module $V(\lambda)$ is the element of $\mathbb {Z}[P]$ given by:
$$\ch_{\mathfrak{g}}V(\lambda) = \sum_{\mu \in P}\dim V(\lambda)_{\mu} e(\mu),\quad \quad \text{where} \, \, V(\lambda)_\mu = \{ v\in V(\lambda) : hv=\mu(h)v\, \forall h\in \mathfrak h\}.$$

\subsection{} Let $\mathfrak g[t]$ be the current Lie algebra of $\mathfrak g$. As a vector space $\mathfrak g[t]:= \mathfrak{g} \otimes \mathbb {C}[t]$, where $\mathbb {C}[t]$ denotes the polynomial ring in indeterminate $t$. With Lie  bracket operation defined by $$ \quad  [a\otimes f, b\otimes g]=[a,b]\otimes fg, \quad \quad \forall \, a, b \in \mathfrak{g}, \, \, f,g \in \C[t],$$ 
$\mathfrak g[t]$ is a $\mathbb Z_+$-graded Lie algebra, where the $\mathbb Z_+$ grading is determined by the degree of the polynomials in $\C[t]$. 
Let $\bold U(\mathfrak g[t])$ be the universal enveloping algebra of $\mathfrak{g}[t]$. With the $\mathbb Z_+$-grading inherited from $ \mathfrak{g}[t]$, $\bold U(\mathfrak g[t])$ is $\mathbb Z_+$-graded. With respect to this grading, we say $(x_{\alpha_1}\otimes t^{r_1})(x_{\alpha_2}\otimes t^{r_2})\cdots(x_p\otimes t^{r_p}) \in \bold U(\mathfrak{g}[t])$ has grade $r_1+r_2+\cdots+r_p$.
For a positive integer $s$, we denote by $\bold{U}(\mathfrak{g}[t])[s]$ the subspace of $\bold{U}(\mathfrak{g}[t])$ spanned by the $s$-graded elements.

\subsection{} A graded representation of $\mathfrak g[t]$ is a $\mathbb Z_+$-graded vector space $V= \underset{r\in \mathbb Z_+}{\bigoplus} V[r]$ such that $$ \bu(\mathfrak g[t])[s].V[r]\subseteq V[r+s], \, \, \forall\, r,s\in \mathbb Z_+.$$ If $U, V$ are two graded representations of $\mathfrak g[t]$, we say $\psi: U\rightarrow V$ is a  graded $\mathfrak g[t]$-module morphism if  $\psi(U[r])\subseteq V[r]$ for all $r\in \mathbb Z_+$. For $s\in \mathbb  Z$, let $\tau_s$ be the grade-shifting operator given by 
$$\tau_s(V)[k] = V[k-s] \quad \quad \forall\, \, k\in \mathbb Z_+,$$ for a  graded representation $V$ of $\mathfrak g[t]$. Given  a $\mathfrak g$-module $U$ and  $z\in \mathbb C$, we can regard $U$ as a $\mathfrak g[t]$ module via the map $\ev_z: \mathfrak g[t]\rightarrow \mathfrak g$ given by $x\otimes t^r\mapsto z^r x$. Let  $\ev_z U$ denote the corresponding evaluation module for $\mathfrak g[t]$. Clearly, for $z=0, \ev_0(U)$ is a graded $\mathfrak g[t]$-module with $\ev_0 U[0]=U$ and $\ev_0 U[r] =0$ for $r>0$.

Let $q$ be an indeterminate. Since in a graded $\mathfrak g[t]$-module $V= \underset{r\in \mathbb Z_+}{\bigoplus} V[r]$, each graded piece $V[r]$, is a $\mathfrak g$-module, the graded character of a module $V$ is given by :
$$ \ch_{gr} V = \sum\limits_{r \geq 0} q^r \ch_{\mathfrak g} V[r] .$$

\subsection{}\label{tensor section} If $V, W$ are two graded representations of $\mathfrak g[t]$, then $V \otimes W$ is a graded $\mathfrak g[t]$-module with the action defined as follows:
$$(X \otimes t^r)(v \otimes w)= (X \otimes t^r)(v) \otimes w + v \otimes (X \otimes t^r)w, \quad \forall \quad X \in \mathfrak g, v \in V, w \in W. $$ Given $k \in \mathbb Z_+$, set $$(V \otimes W)[k]= \bigoplus\limits_{i \in \mathbb Z_+} V[i] \otimes W[k-i],$$ with the assumption that $W[j] =0$ if $j<0$.

\subsection{} For $\lambda \in P^+$, the local Weyl module $W_{loc}(\lambda)$ is defined as  the  $\mathfrak{g}[t]$-module generated by non-zero vector $w_{\lambda}$ which satisfies the following conditions: 
\begin{equation}\label{loc.Weyl.2.5}
(\mathfrak n^+ \otimes \C[t])w_\lambda= 0,\quad \,(h_{\alpha_i}\otimes t^s)w_\lambda= \lambda(h_{\alpha_i})\delta_{s,0}w_\lambda, \quad \,
(y_{\alpha_i}\otimes 1)^{\lambda(h_{\alpha_i})+1}w_\lambda =0,
\end{equation} 
for all $i \in I$ and $s \in \mathbb Z_+$. Clearly,  $W_{loc}(\lambda)$ is a graded $\mathfrak{g}[t]$ module, and every finite-dimensional cyclic module with the highest weight $\lambda$ is a quotient of $W_{loc}(\lambda)$. These modules were introduced in \cite{MR1850556} in the context of quantum affine algebras.

\subsection{} Given a finite-dimensional representation $V$ of $\mathfrak g[t]$ and $z\in \mathbb C$, let $V^z$ denote the $\mathfrak g[t]$-module on which the action of $\mathfrak g[t]$ on $V$ is given as follows : 
$$x\otimes t^s .v = (x\otimes (t+z)^s). v, \qquad v\in V, \, s\in \mathbb Z_+.$$ 
Given $s$ graded finite-dimensional cyclic $\mathfrak g[t]$-modules
$V_1,\cdots, V_s$ with generators $v_1,\cdots, v_s$ and $\bold z= (z_1,\cdots, z_s),$ a $s$-tuple of distinct complex numbers, let $$V(\bold{z}) = V_1^{z_1}
\otimes \cdots \otimes V_s^{z_s}.$$ Clearly, $V(\bold{z})$ is a graded cyclic $\mathfrak g[t]$-module generated by $v_1\otimes\cdots\otimes v_s.$ The $\mathbb N$-grading in $\bold{U}(\mathfrak g[t])$ induces a $\mathfrak g$-equivariant grading on $V(\bold z)$ given by 
$$V(\bold z){[k]} = \underset{0\leq r\leq k}{\bigoplus} \bold{U}(\mathfrak g[t])[r]. v_1\otimes\cdots\otimes v_s.$$  The associated graded  $\mathfrak g[t]$-module 
$$ V_1^{z_1}\ast \cdots \ast V_s^{z_s} := V(\bold z){[0]} \oplus \underset{k\in \mathbb N}{\bigoplus} \dfrac{\bold{V}(\bold z){[k]}}{\bold{V}(\bold z){[k-1]}} $$ is called the {\it{fusion product}} of $V_1,\cdots, V_s$ at $\bold{z}$. These modules were introduced in \cite{MR1729359}. It was conjectured that they are independent of the parameters of evaluation $z_1,\cdots,z_s$. 
In this context, the following lemma has proved useful.
\begin{lemma}\label{fusion_def} Suppose that $\bold {V}(\bold\lambda,\boz):=V(\lambda_1)^{z_1}\otimes \cdots \otimes V(\lambda_s)^{z_s}$ is an irreducible $\mathfrak g[t]$-module. 
If for $v\in \bold {V}(\bold\lambda,\boz)$,  $\bar{v}$ denotes its image in $ V(\lambda_1)^{z_1}\ast \cdots \ast V(\lambda_s)^{z_s}$  then
$$ x\otimes t^p.\bar{v} = x\otimes (t-a_1)\cdots(t-a_p). {\bar{v}}, \hspace{.25cm} \forall\ x\in \mathfrak g, \ \text{and}\ 
a_1,\cdots, a_p\in \mathbb C.$$ 
\end{lemma}

\subsection{} Given a dominant integral weight $\lambda$ and a positive integer $\ell$, the level $\ell$ Demazure module with highest weight $\lambda$, 
which we denote by $D(\ell, \lambda)$,  is defined as the graded quotient of $W_{loc}(\lambda)$ by the submodule generated by the elements $$\{(y_\alpha\otimes t^p)^{r+1}w_\lambda :p\in \mathbb Z_+, \, r\geq \max\{0, \lambda(h_\alpha)-\ell p\},  \text{ for } \alpha\in \Phi^+\}.$$  It was shown in \cite{MR2271991} for $\mathfrak g$ of type A and in \cite{MR2323538} for $\mathfrak g$ of type ADE, that the $\mathfrak{g}[t]$-module $W_{loc}(\lambda)$ is isomorphic to $D(1,\lambda)$.

%
%
%
%
%
%
%

\section{Chari-Loktev Basis of $W_{loc}(\lambda)$ and POP}

\subsection{} In \cite{MR2271991}, Chari and Loktev introduced a basis for local Weyl modules of the current algebras associated to special linear Lie algebras and proved the following.
\begin{lemma} \label{CL.W(lambda).dim}
  Given  $\mu =\sum\limits_{i=1}^n m_i \omega_i\in P^{+}$,  
    $\dim W_{loc}(\mu ) = \prod\limits_{i=1}^{n} {n+1 \choose i}^{m_{i}}$
\end{lemma}

\noindent Subsequently, partition overlaid patterns, POPs for short, were introduced in \cite{pop}, and it was shown that they form a convenient parametrizing set for the Chari-Loktev basis. We now recall the definitions and results in this direction that are revelant for this paper. 

\begin{definition} Given a partition $p_\lambda=(\lambda_1 \geq \lambda_2 \geq \cdots \geq \lambda_n\geq \lambda_{n+1})$, a Gelfand-Tsetlin (GT) pattern with bounding sequence $p_\lambda$  is an array of row vectors  $\lambda= [\lambda^{i}_j]_{1\leq i\leq n+1, 1\leq j\leq i}$ 
    $$\begin{array}{ccccccc}
       & & & \lambda_{1}^1  & & &   \\
       & & \lambda_{1}^2  & &  \lambda_{2}^2 & &\\
       &\cdots&  & \cdots &  & \cdots&\\
       \lambda_{1}^n& & \lambda_{2}^n & &\cdots & &\lambda_{n}^n\\
 \lambda_{1}^{n+1}\, \quad & &  \lambda_{2}^{n+1} & &\cdots  & & \, \, \qquad \lambda_{n+1}^{n+1}
    \end{array}$$
where $\lambda_j^{n+1} = \lambda_j$ for $1\leq j\leq n+1$, that is, the last row coincides with the bounding sequence $p_\lambda$  and  the entries of the $j^{th}$ row $\underline{\lambda}^j =(\lambda_1^j,\cdots,\lambda_j^j)$, for $1\leq j\leq n$, 
 satisfy the condition:
    $$ \lambda_{k}^i-\lambda_{k}^{i-1} \in \mathbb Z_+, \quad \lambda_{k}^{i-1}-\lambda_{k+1}^i \in \mathbb Z_+, \quad 1\leq i\leq n,\, 1\leq k\leq i-1.$$ 
\end{definition}

\noindent For a positive integer $m$, set $[m]=\{1,2,\cdots,m\}$.

\begin{definition} A partition overlaid pattern (POP, for short) consists of an integral GT pattern $\underline{\lambda}^1, ..., \underline{\lambda}^n$, and, for every ordered pair $(j, i)\in [n]\times [j]$, a partition $\pi(j)^i$ such that $\pi(j)^i$ has $\lambda_{i}^{j+1}-\lambda_{i}^j$ parts and length of each part is less than equal to $\lambda_{i}^j-\lambda_{i+1}^{j+1}$. We shall denote a partition overlaid pattern by a pair $(\lambda, \pi_\lambda)$, where $\lambda$ denotes a GT pattern and $\pi_\lambda = (\pi_\lambda(j)^i)_{j\in [n], i\in [j]}$ is an array of the partitions that satisfy the conditions for POP associated to $\lambda$.  \end{definition}

\subsection{} Given a partition overlaid pattern $(\lambda, \pi_\lambda)$, for $(j,i)\in [n]\times [j] $ and $0\leq s\leq \lambda^j_i-\lambda^{j+1}_{i+1}$,
 set
 $$r_{\lambda}^{ij}(s) = \text{ the number of parts of } \pi_\lambda(j)^i \text{ that are equal to } s,$$ 
and define elements $y_{ij}(\lambda,\pi_\lambda), y_{j}(\lambda,\pi_\lambda), y(\lambda,\pi_\lambda)
\in \bu(\mathfrak n^-[t])$ as follows :
$$y_{ij}(\lambda, \pi_\lambda) := (y_{ij}\otimes 1)^{r^{ij}_\lambda(0)}  (y_{ij}\otimes t)^{r^{ij}_\lambda(1)} \cdots (y_{ij}\otimes t^{ \lambda^j_i-\lambda^{j+1}_{i+1}})^{ r_{\lambda}^{ij}( \lambda^j_i-\lambda^{j+1}_{i+1})}, $$
\begin{align*} & y_j(\lambda,\pi_\lambda) =  y_{1j}(\lambda,\pi_\lambda) y_{2j}(\lambda,\pi_\lambda)\cdots y_{jj}(\lambda,\pi_\lambda),\\
&y(\lambda, \pi_\lambda) = y_1(\lambda,\pi_\lambda)y_2(\lambda,\pi_\lambda)\cdots y_{n}(\lambda,\pi_\lambda).\end{align*}

\vspace{.35cm}

\subsection{The set $\mathfrak P(\mu)$ and an order on $\mathfrak P(\mu)$} Given $\mu = \sum\limits_{i=1}^n \mu_i\omega_i\in P^+$, let $\mathfrak P(\mu)$ be the set of all partition overlaid patterns $(\eta, \pi_\eta)$ such that $\eta$ is a GT pattern with bounding sequence $b_\mu = (\sum\limits_{i=1}^n\mu_i, \sum\limits_{i=2}^n\mu_i,\cdots, \mu_n,0 )$ and $\pi_\eta = (\pi_\eta(j)^i)_{j\in [n], i\in [j]}$ is an array of partitions that satisfy the conditions for POP associated to $\eta$.

\noindent For each pair of integers $(j,i)\in [n]\times [j]$, define a function $|.|_{ij}: \mathfrak P(\mu)\rightarrow \mathbb Z$ such that
$$|(\eta,\pi_\eta)|_{ij}= 
|\pi_\eta(j)^i| = \sum\limits_{s=0}^{\eta_i^j-\eta^{j+1}_{i+1}}r_\eta^{ij}(s) $$ 
 where $ r^{ij}_\eta(s)$ is as defined in \secref{POP.basis}.\\ Given $(\eta,\pi_\eta), (\zeta,\pi_\zeta) \in \mathfrak P(\mu)$, we say $(\eta,\pi_\eta)\geq (\zeta,\pi_\zeta)$ if \\
\begin{enumerate}
\item[i.] there exists a pair $(k,l)\in[n]\times[k]$ such that $|(\eta,\pi_\eta)|_{kl} < |(\zeta,\pi_\zeta)|_{kl}$ and  \\
$|(\eta,\pi_\eta)|_{ij} =|(\zeta,\pi_\zeta)|_{ij}$ for all pairs $(i,j)\in\{(k,p): p>l\}\cup \{(q,p): q>k\}$. 
\vspace{.15cm}

\item[ii.] or $|(\eta,\pi_\eta)|_{ij} =|(\zeta,\pi_\zeta)|_{ij}$ for all $(j,i)\in [n]\times[j]$ and  there exists  $(k,l)\in [n]\times[k]$  and $0\leq s\leq \eta^l_k-\eta^{l+1}_{k+1}$ such that $r_\eta^{kl}(s) > r_\zeta^{kl}(s)$ and $r_\eta^{ij}(s) = r_\zeta^{ij}(s)$ for all pairs $(i,j)\in\{(k,p): p>l\}\cup \{(q,p): q>k\}$ and all $s$. 
\end{enumerate}
Note that this makes $\mathfrak P(\mu)$ a totally ordered set. Hence given $(\eta,\pi_\eta) \in \mathfrak P(\mu)$ there exists a unique element $(\eta^+,\pi_{\eta^+}) \in \mathfrak P(\mu)$ such that $(\eta^+,\pi_{\eta^+}) > (\eta, \pi_\eta)$ and for all $(\zeta, \pi_\zeta) \in \mathfrak P(\mu)$ satisfying the condition $(\zeta, \pi_\zeta) > (\eta, \pi_\eta)$, we have $(\zeta, \pi_\zeta) \geq (\eta^+,\pi_{\eta^+}).$

\label{POP}

\subsection{} With regard to the order introduced on $\mathfrak P(\mu)$ the following result follows from \cite[Theorem 4.3]{pop}.

 \begin{proposition}\label{popth} Given $\mu = \sum\limits_{i=1}^n \mu_i\omega_i\in P^+$, let $\mathfrak P(\mu)$ be the set of all partition overlaid patterns $(\eta, \pi_\eta)$ as defined above. 
Then, $$ \mathfrak B(\mu) = \{ y(\eta,\pi_\eta)w_\mu : (\eta,\pi_\eta)\in \mathfrak P(\mu)\} $$ forms an ordered basis  of the $\mathfrak{sl}_{n+1}[t]$-module $W_{loc}(\mu),$ where the order on $ \mathfrak B(\mu)$ is inherited from the order on $\mathfrak P(\mu)$ .
 \end{proposition}

\begin{remark} Let $w_0$ denote the longest element of the Weyl group of $\mathfrak g = \mathfrak{sl}_{n+1}$. From \thmref{popth} it is clear that for  $\mu = \sum\limits_{i=1}^n \mu_i\omega_i$,  the vector $w_{w_0\mu}$ given by:
\begin{equation} \label{w_w_0} w_{w_0\mu} := (y_{11})^{(m_n)}(y_{12})^{(m_{n-1})}(y_{22})^{m_{n}}\cdots (y_{nn})^{(m_n)}w_\mu \in \mathfrak B(\mu)\end{equation} is the unique element of $W_{loc}(\mu)$ of weight $w_0\mu$.
Notice that for $1\leq i\leq j\leq n$, the exponent of $y_{ij}$ in the POP-expansion of $w_{w_0\mu}$ is $m_{n-j+i}$. Setting, 
 $$w_\mu^{(i+1j)} := (y_{i+1j})^{m_{n-j+i+1}}\cdots (y_{jj})^{m_n}(y_{1j+1})^{m_{n-j}}\cdots (y_{nn})^{m_n}w_\mu $$ in $W(\mu)$, observe that the $h_{\alpha_{ij}}$-weight of the element $w_\mu^{(i+1j)}$ is 
 $$\mu(h_{\alpha_{ij}}) -\sum\limits_{l=j+1}^{n} \sum\limits_{k=1}^{l} m_{n-l+k}\alpha_{kl}(h_{\alpha_{ij}}) - \sum\limits_{p=i+1}^j m_{n-j+p}\alpha_{pj}(h_{\alpha_{ij}}) = m_{n-j+i}.$$
\end{remark}
\label{POP.basis}

%
%
%
%
%
%
%
%
%
%
\section{CV-Realization and tensor product of Local Weyl modules } 

In \cite{MR3296163}, a family of graded representations of the current Lie algebras, which we refer to as CV-modules,  was introduced. It has been shown that these modules subsume many disparate classes studied earlier. In particular, these are related to Weyl modules, higher-level Demazure modules, Kirillov-Reshetikhin modules and truncated Weyl modules \cite{MR3296163, MR3407180}. These play an important role in the proof of the main results of this paper. In this section we recall their definition and properties. 

\subsection{} The following definition of CV-modules was given in \cite{MR3296163}.\\

\noindent  For a non-zero dominant integral weight $\mu \in P^{+}$, a  $|\Phi^{+}|$-tuple of partitions  $\boldsymbol{\xi} = ( \xi^{\alpha})_{\alpha \in \Phi^{+} }$ is said to be $\mu $-compatible,  if 
    $$\xi^{\alpha} = (\xi_{1}^{\alpha} \geq \xi_{2}^{\alpha} \geq \dots \geq \xi_{s}^{\alpha} \geq \dots \geq 0)_{\alpha \in \Phi^+}, \quad |\xi^{\alpha}| = \sum_{j \geq 1} \xi_{j}^{\alpha} = \mu(h_{\alpha}).$$
 \begin{definition} \label{CV definition}  Given $\mu\in P^+$ and a $|\Phi^+|$-tuple of $\mu$-compatible partition $\boldsymbol{\xi} = ( \xi^{\alpha})_{\alpha \in \Phi^{+} }$, the CV-module  $V(\boldsymbol{\xi})$ is defined as the graded quotient of $W_{loc}(\mu)$ by the submodule generated by the graded elements 
$$(x_{\alpha}\otimes t )^s (y_{\alpha}\otimes 1)^{r+s} w_\mu = 0 ; \quad \forall \, \alpha \in \Phi^{+} , s,r \in \mathbb{N} , \quad r+s \geq 1+rk+\sum_{j \geq k+1} \xi_{j}^{\alpha}, \text{for some } k \in \mathbb{N}.$$  \end{definition}

 \subsection{} Define a power series in the indeterminate $u$ as follows. For $\alpha \in \Phi^{+}$ let
	\begin{eqnarray*}
	H_{\alpha}(u) = \exp \left(- \sum_{r=1}^\infty \frac{h_{\alpha}\otimes t^r}{r} u^r\right)
	\end{eqnarray*}
	\\
	Let $P_{\alpha}(u)_k$ denote the coefficient of $u^k$ in  $H_\alpha(u)$. 
	For $s,r,k \in \mathbb{Z}_{+}$ and $x\in \mathfrak{g}$, define elements $x(r,s), {_k}x(r,s) \in U(\mathfrak{g}[t])$ as : 
	$$ x(r,s) = \sum\limits_{(b_{p})_{p \geq 0}} (x \otimes 1)^{(b_{0})}(x \otimes t^{1})^{(b_{1})}   \cdots (x \otimes t^{s})^{(b_{s})}$$
$$ {_k}x(r,s) = \sum\limits_{(b_{p})_{p \geq k}} (x \otimes t^k)^{(b_{k})}(x \otimes t^{k+1})^{(b_{k+1})}   \cdots (x \otimes t^{s})^{(b_{s})}$$	
where $p , b_{p} \in \mathbb{Z}_{+}$ are such that $ \sum\limits_{p}b_{p} = r $, $\sum\limits_{p}pb_{p} = s$ and for any $X \in \mathfrak{g}$, $X^{(p)} = \dfrac{X^{p}}{p!}$. The following result was proved in  \cite{MR0502647} and reformulated in its present form in \cite{MR1850556}. 
	
\begin{lemma} \label{Garland equation} Given $s\in N, r \in \mathbb{Z}_{+}$, we have 
\begin{equation}\label{G.eq}
(x_{\alpha}\otimes t)^{(s)}(y_{\alpha}\otimes 1)^{(r+s)} - 
(-1)^s \big(\ \underset{k\geq 0}{\sum}\  y_{\alpha}(r,s-k){P_{\alpha}(u)}_k\big) 
\in \bu(\mathfrak g[t])\mathfrak n^+[t].
		\end{equation} 	\end{lemma}   

 Denoting the image of $w_\mu$ in $V(\boldsymbol{\xi})$ by $v_{\boldsymbol{\xi}}$, 
and using  \lemref{Garland equation}, the CV module $V(\boldsymbol{\xi})$ can be defined  as a graded $\mathfrak g[t]$-module generated by a non-zero vector $v_\bxi$ which satisfies the following relations: 
\begin{align} 
&(x_{\alpha}\otimes t^r )v_\bxi  = 0; \quad &\forall\, \alpha \in \Phi^{+}, r\geq 0 \\
&(h \otimes t^s)v_{\bxi}  = \delta_{s,0}\mu(h) v_{\bxi}; \quad &\forall\,  h \in \mathfrak{h}, s \in Z_{+}\\
&(y_{\alpha_{i}} \otimes 1)^{\mu(h_{\alpha_{i}})+1} v_{\bxi} = 0 ; \quad &\forall\, i \in I \\
\label{xrs relation} 
&y_{\alpha}(r,s)v_{\bxi} = 0, \quad &\forall\, \alpha\in \Phi^+, \text{when } r+s \geq 1+rk+ \sum_{j \geq k+1}\xi_{j}^{\alpha}.
\end{align}
It was shown in  \cite[Proposition 2.6]{MR3296163} that relation\eqref{xrs relation} in $V(\xi)$ is equivalent to the following:
\begin{align} \label{xrs relation.2} 
&{_k}y_{\alpha}(r,s)v_{\bxi} = 0, \quad &  \quad &\forall\, \alpha\in \Phi^+, \quad \text{when } r+s \geq 1+rk+ \sum_{j \geq k+1}\xi_{j}^{\alpha}.
\end{align}

\noindent Given $\lambda=\sum\limits_{i=1}^n m_i\omega_i$, set $m_{ij} = \sum\limits_{k=i}^j m_k$ and   $\xi_\lambda^{\alpha} = (1^{m_{ij}})$, for $\alpha=\alpha_{ij}\in \Phi^+$, $1\leq i\leq j\leq n$. It has been shown that the local Weyl module $W_{loc}(\lambda)$ is isomorphic to the CV module $V(\bxi_\lambda)$, where 
$\bxi_\lambda= (\xi_\lambda^\alpha)_{\alpha\in \Phi^+}$. Using the $CV$-module 
realization of local Weyl modules, the following was 
proved in \cite{24}.
\begin{lemma}\label{sl2.CV.op} For $n\in \mathbb Z_+$, let $W_{loc}(n)$ be the local Weyl module for $\mathfrak{sl}_2[t]$ with highest weight $n\omega_1$. If $w_n$ is the highest weight generator of $W_{loc}(n)$, then $w_{-n}: = (y_{\alpha}\otimes 1)^{n}w_n$ is a non-zero vector in $ W_{loc}(n)$ such that :
$$\begin{array}{lll}
	y_{\alpha}\otimes t^r. w_{-n} &= 0,  \quad &\forall\, r\geq 0,\\
        h_\alpha \otimes t^r. w_{-n} &= \delta_{0,r} (-n) w_{-n}, \quad &\forall\, r\geq 0,\\
	x_{\alpha}(r,s). w_{-n} &=0,  \quad & \text{if } r+s \geq 1+rl+n-l .   
\end{array}$$ Further, $W_{loc}(n)$ is isomorphic to 
$\bu(\mathfrak n^+[t]).w_{-n}$ as a $\mathfrak{sl}_2[t]$ module. \end{lemma}

\noindent Denoting $x_{\alpha_{ij}}, y_{\alpha_{ij}}$ in $\mathfrak{sl}_{n+1}$ by $x_{ij}, y_{ij}$, respectively, note that for $\alpha=\alpha_{ij}\in \Phi^+$, we have
\begin{equation}\label{[y_i_j]}
[y_{rs}, y_{ij}]= \left\{ \begin{array}{ll} 
-y_{rj}, & \text{if } r\leq s=i-1\leq j \\
\, 0, & \text{if } s+1<i,
\text{ or } i\leq r\leq s\leq j, \text{ or } r\leq j\leq s \end{array}\right. \end{equation} 

\noindent Hence we get the following as a corollary of \lemref{sl2.CV.op}.

\begin{corollary} \label{lowest.wt.W.mu}For $\mu = m_1 \omega_{1} +m_2 \omega_{2} + \dots + m_n \omega_n \in P^{+}$ let $w_{w_0\mu}$ be the lowest weight vector of $W_{loc}(\mu)$ as defined in \eqref{w_w_0}. Then for all $\alpha\in \Phi^+,$ 
 \begin{align*}
      y_{\alpha}\otimes t^r. w_{w_{0}\mu} & = 0 \quad \forall r\geq 0 \\
      h_{\alpha} \otimes t^r.  w_{w_{0}\mu} & = \delta_{r,0}\,  w_{0}\mu(h_{\alpha}) 
w_{w_{0}\mu}  \qquad \forall\, r \geq 0 \\
      x_{\alpha}(r,s). w_{w_{0}\mu} & = 0, \quad  \text{if } r+s \geq 1+rk-w_{0}(\mu)(h_{\alpha})-k
  \end{align*}
\end{corollary}
\proof Using \eqref{[y_i_j]} and the notation introduced in \secref{POP} , 
observe that $$y_{qs}\otimes t^r w_{w_0\mu} \in \sum\limits_{k=1}^{q} \bu(n^-[t])  y_{ks}^{m_{n-s+k}}(y_{ks}\otimes t^r). w_{\mu}^{k+1s}.$$ Since $h_{\alpha_{ks}}$-weight of  
$w_{\mu}^{k+1s}$ is $m_{n-k+1}$, using \lemref{sl2.CV.op}, we see that 
$$y_{qs}\otimes t^r w_{w_0\mu} =0, \qquad \forall \, r\geq 0. $$
Similar arguments show that 
$h\otimes t^r.w_{w_0\mu} =0 $ for all $h\in \mathfrak h$ and $r\in \mathbb N.$ 
Finally observe that for each $\alpha\in \Phi^+$, the $h_\alpha$-weight of 
$w_{w_0\mu}$ is $w_0\mu(h_\alpha) = \mu(h_\alpha) -\sum\limits_{k=1}^n\sum\limits_{j=1}^k m_{n-k+j}\alpha_{jk}(h_\alpha)$. Hence applying \lemref{sl2.CV.op}, for 
the $\mathfrak{sl}_2[t]$ subalgebra of $\mathfrak g[t]$ corresponding to the 
root $\alpha$, we see that $$x_\alpha(r,s) w_{w_0\mu} =0,  \quad  
\text{if } r+s \geq 1+rp-w_{0}(\mu)(h_{\alpha})-p\qquad \forall \, r\geq 0. $$ 
Hence the result follows. \endproof

\begin{remark}\label{lowest.wt.W.mu.rem} Applying the $\mathfrak{sl}_{n+1}[t]$-isomorphism that maps $$y_{\alpha}\otimes t^s \mapsto x_{\alpha}\otimes t^s , \quad x_{\alpha}\otimes t^s \mapsto y_{\alpha}\otimes t^s, \quad h_{\alpha}\otimes t^s \mapsto -h_{\alpha}\otimes t^s, \quad \forall s \geq 0, \, \alpha\in \Phi^+,$$ on the generators and relations of the CV-module realization of the local Weyl module $W_{loc}(\mu)$, we observe that by \corref{lowest.wt.W.mu} there exists an $\mathfrak{sl}_{n+1}[t]$-module isomorphism between  $W_{loc}(\mu)$ and $\bu(\mathfrak n^+[t])w_{w_0\mu}.$ 
\end{remark}
 
\subsection{} We now consider the tensor product of two local Weyl modules.

\begin{lemma}\label{generator of tensor}    For $\lambda =\sum\limits_{i=1}^n m_i \omega_i , \mu = \sum\limits_{i=1}^n k_i \omega_i$ in $P^{+}$, let $W_{loc}(\lambda)$ and $W_{loc}(\mu)$ be the local Weyl modules with highest weights $\lambda$ and $\mu$ respectively. Then,
\begin{equation} \label{dimension of tensor}
        \dim W_{loc}(\lambda) \otimes W_{loc}(\mu) = \prod_{i=1}^{n} {n+1 \choose i}^{m_{i}+k_i }
    \end{equation} 
Further, as a $\mathfrak{sl}_{n+1}[t]$-module $W_{loc}(\lambda) \otimes W_{loc}(\mu)$ is generated by the vector $w_{\lambda}\otimes w_{w_0\mu}$, where $w_\lambda$  is the highest weight generator of $W_{loc}(\lambda)$ and $w_{w_0\mu}$ the lowest weight generator of $W_{loc}(\mu)$ (as defined in \eqref{w_w_0}).   
\end{lemma}
 \begin{proof} Note, \eqref{dimension of tensor} follows immediately from \lemref{CL.W(lambda).dim}. \\

\noindent Let $S$ be the submodule of $W_{loc}(\lambda) \otimes W_{loc}(\mu)$ generated by the vector $w_{\lambda}\otimes w_{w_0\mu}$. As $\mathfrak n^{-}[t]. w_{w_0\mu} = 0$, 
   $$ U(\mathfrak n^{-}[t])(w_{\lambda}\otimes w_{w_0\mu}) = U(\mathfrak n^{-}[t])w_{\lambda}\otimes w_{w_0\mu} = W_{loc}(\lambda) \otimes w_{w_0\mu} \subset S.$$  Hence for $r\in \mathbb{N}$ and $v \in W_{loc}(\lambda)$ we have,
   $$ x_{\alpha} \otimes t^r. v \otimes w_{w_0\mu} + v\otimes (x_{\alpha} \otimes t^r). w_{w_0\mu} = x_{\alpha} \otimes t^r. (v \otimes w_{w_0\mu}) \in S.$$
However, $ x_{\alpha} \otimes t^r. v \otimes w_{w_0\mu}\in W_{loc}(\lambda)\otimes w_{w_0\mu}\subset S$. It therefore follows that  $$ v\otimes (x_{\alpha} \otimes t^r). w_{w_0\mu} \in S, \qquad  \forall \, v\in W_{loc}(\lambda).$$ Now using the realization, $W_{loc}(\mu) = U(\mathfrak n^+[t])w_{w_0\mu}$ (\remref{lowest.wt.W.mu.rem}) and applying induction on the length of $X \in \bu(\mathfrak n^{+}[t])$ for all the basis elements $Xw_{w_0\mu}$ of $W_{loc}(\mu)$, we see that $W_{loc}(\lambda)\otimes W_{loc}(\mu) \subseteq S$. Thus the lemma follows. \end{proof}

%
%
%
%
%
%
%
%
%

\section{Pieri formulas}\label{pieri formulas}

\subsection{} In \cite{MR3443860}, I.G. Macdonald introduced a family of orthogonal symmetric polynomials  $P_{\lambda}(\boldsymbol{x}; q, t)$,  $\boldsymbol{x}= (x_1 , x_2 , \dots , x_{n+1})$ which form a basis for the ring of symmetric polynomials in $\mathbb{C}(q,t)[x_1 , \dots , x_{n+1}]$, where $\lambda$ is a partition with at most $n+1$ parts.  It was proved  in \cite{MR1771615},  that the character of level $1$ Demazure module of type $A$  is a specialized Macdonald polynomial and this result was  extended to other cases in \cite{MR1953294}.

Given a partition $(m_1,m_2,\cdots,m_{n+1})$ or equivalently $\lambda = \sum\limits_{i=1}^n(m_i-m_{i+1})\omega_i$, a dominant integral weight of $\mathfrak{sl}_{n+1}[t]$, it is known from results in \cite{MR2271991,MR1771615} that 
\begin{equation} \label{CL.ch} P_{(m_1,m_2,\cdots,m_{n+1})} (\boldsymbol{x};q,0) = \ch_{gr} W_{loc}(\lambda)
\end{equation} 
As Pieri formulas are used to determine products of Macdonald polynomials, in this section, we use them to obtain the graded character of the tensor product of two Weyl modules  of $\mathfrak{sl}_{n+1}[t]$.

\begin{lemma} \label{loc in loc} For $m,k\in \mathbb N$, the following hold :
\begin{itemize}
\item[i.]     
$$\ch_{gr} W_{loc}(m\omega_{1})\otimes W_{loc}(k\omega_{n}) = \sum\limits_{i=0}^{\text{min}\{m,k\}} \begin{bmatrix} m \\ i \end{bmatrix}_{q} \begin{bmatrix} k \\ i \end{bmatrix}_{q}(1-q)\cdots (1-q^{i}) \ch_{gr} W_{loc}((m-i)\omega_{1} + (k-i) \omega_{n})$$
\item[ii.] $$\ch_{gr} W_{loc}(m\omega_{1}) \otimes W_{loc}(k\omega_{1}) = \sum\limits_{i=0}^{\text{min}\{m,k\}} \begin{bmatrix} m \\ i \end{bmatrix}_{q} \begin{bmatrix} k \\ i \end{bmatrix}_{q}(1-q)\cdots (1-q^{i}) \ch_{gr} W_{loc}((m+k-2i)\omega_{1} + i \omega_{2})$$
\item[iii.] $$\ch_{gr} W_{loc}(m\omega_{n}) \otimes W_{loc}(k\omega_{n}) = \sum\limits_{i=0}^{\text{min}\{m,k\}} \begin{bmatrix} m \\ i \end{bmatrix}_{q} \begin{bmatrix} k \\ i \end{bmatrix}_{q}(1-q)\cdots (1-q^{i}) \ch_{gr} W_{loc}(i\omega_{n-1} + (m+k-2i) \omega_{n})$$
\end{itemize}
\end{lemma}

\vspace{.35cm}

\subsection{} Set
 \begin{align*}
	(q;q)_{n} &= (1-q)(1-q^2 ) \dots (1-q^n ) \quad  \\
	(t;q)_n &= (1-t)(1-tq) \dots (1-tq^{n-1}) \quad \text{for } |q|<1 \\
	(t;q)_{\infty} &= \prod_{i=0}^{\infty}(1-tq^{i})
\end{align*}
Let  $g_{m}(\boldsymbol{x};q,t)$ denote the coefficient of $y^m $ in the power 
series expansion of the infinite product 
$$\prod_{i\geq 1} \frac{(tx_{i} y;q)_{\infty}}{(x_{i}y;q)_{\infty}} = \sum_{m \geq 0}g_{m}(\boldsymbol{x};q,t)y^{m}.$$
For any positive integer $m$, it has been shown in  \cite[Chapter VI, Equation (4.9)]{MR3443860} that : 
 \begin{equation}\label{gm equation}
P_{(m)}(\boldsymbol{x};q,t)= \frac{(q;q)_{m}}{(t;q)_{m}}g_{m}(\boldsymbol{x},q,t),
\end{equation}
and hence, 
\begin{equation} \label{P in terms of g}
P_{(m)}(\boldsymbol{x};q,0)= (q;q)_{m}g_{m}(\boldsymbol{x},q,0)
\end{equation}
\vspace{.15cm}

\noindent  Given a partition $\rho = (\rho_1\geq \rho_2\geq \cdots\geq \rho_{n+1} )$,  
the left-justified diagram having $n+1$ rows and $\rho_i$ boxes in the $i^{th}$ row for $1\leq i\leq n+1$, is said to be the diagram of the partition of $\rho$. For each box $s = (i,j)$ in the diagram of $\rho$,  let 
$$a_{\rho}(s) = \rho_i -j, \qquad l_{\rho}(s) = \rho_{j}' -i,$$ where $\rho_j'$ is equal to the number of boxes in the $j^{th}$ column of the diagram of $\rho$. We now  
recall the result on the Pieri formulas that is used to  compute the product of the specialized Macdonald polynomials.  
\begin{lemma}\cite[Theorem 6.24]{MR3443860} \label{PRule} 
   Let $m\in \mathbb Z_{+}$. Given partitions $\lambda$ and $\mu$ such that  
$\lambda \supset \mu$ and $\lambda - \mu = (m)$, we have,
    $$P_{\mu}(\boldsymbol{x};q,0) g_{m}(\boldsymbol{x};q,0) = \sum_{\lambda - \mu = (m)} \phi_{\lambda/ \mu} P_{\lambda}(\boldsymbol{x};q,0)$$ 
and the coefficients are given by 
$$ \phi_{\lambda / \mu} = \prod_{s \in C_{\lambda / \mu}} \frac{b_{\lambda}(s)}{b_{\mu}(s)}$$ 
where  $C_{\lambda /\mu}$ denotes the union of columns that intersect $\lambda-\mu$, and 
$$ b_{\lambda}(s) = b_{\lambda}(s;q,t) = \begin{cases} 
	 \frac{1-q^{a_{\lambda}(s)}t^{l_{\lambda}(s)+1}}{1- q^{a_{\lambda}(s)+1}t^{l_{\lambda}(s)}} & \quad \text{if } s \in \lambda, \\
       1 & \quad \text {otherwise} 
\end{cases} $$
\end{lemma}

\subsection{Proof of \lemref{loc in loc}.} 
\begin{proof} 
(i). As \eqref{CL.ch} holds, we consider the product $P_{(m)}(\boldsymbol{x};q,0)P_{(k,k,\dots,k,0)}(\boldsymbol{x};q,0)$ to prove (i). Using \lemref{PRule} and \eqref{P in terms of g}, we have 
$$\begin{array}{l}
P_{(m)}(\boldsymbol{x};q,0)P_{(k,k,\dots,k,0)}(\boldsymbol{x};q,0) = (q;q)_{m}P_{(k,k,\dots,k,0)}(\boldsymbol{x};q,0)g_{m}(\boldsymbol{x};q,0) \\
\\
				= (q;q)_{m} \sum\limits_{\lambda - (k,k,\dots,k,0) = (m)} \phi_{\lambda/ (k,k,\dots,k,0)}(q,0) P_{\lambda}(\boldsymbol{x};q,0)\\ \\
				= (q;q)_{m} \sum\limits_{\lambda - (k,k,\dots,k,0) = (m)}\, \prod\limits_{s \in C_{\lambda / (k,k,\dots,k,0)}} \dfrac{b_{\lambda}(s;q,0)}{b_{(k,k,\dots,k,0)}(s;q,0)}\, P_{\lambda}(\boldsymbol{x};q,0) \end{array}$$ 

\noindent Observe  that $\lambda-(k,k,\cdots,k,0)=(m)$ if and only if $\lambda = (k+m-i,k,k\cdots,k,i)$ for some $i\leq \min\{m, k\}$ and for the partition 
$\lambda(i)=(k+m-i,k,k\cdots,k,i)$, we have,    
$$ C_{\lambda(i)/(k,k,\cdots,k,0)} = \{(1,j), \text{ for } k\leq j\leq m-i\} \cup \{(r,l): 1\leq r\leq n+1, 1\leq l\leq i\}.$$
By the formula, $b_\lambda(s;q,0)=1$ whenever $l_\lambda(s)\neq 0$. Since 
$$\begin{array}{ll}
l_{\lambda(i)}(s) >0, \quad &\text{ for } s=(r,j),\,  1\leq r\leq n,\\
l_{(k,k,\cdots,k,0)}(s) >0, \quad &\text{ for } s=(r,j), \, 1\leq r\leq n-1,
\end{array}$$ 
and 
 $$(a_{\lambda(i)}(s), l_{\lambda(i)}(s))  = \left\{ \begin{array}{ll} (m-i+k-k-r,0) & \text{ for } s=(1,k+r),\\ (i-l,0) & \text{ for } s=(n+1,l)\\
 \end{array} \right. $$
$$(a_{(k,k,\cdots,k,0)}(s), l_{(k,k,\cdots,k,0)})  =  (k-l,0) \quad \text{ for } s=(n,l)\, 1\leq l\leq i, $$
it follows that 
$$\phi_{\lambda(i)/(k,k,\cdots,k,0)} =  \dfrac{(1-q^k)(1-q^{k-1} ) \dots (1-q^{k-i+1} )}{(1-q)(1-q^2 ) \dots (1-q^{m-i} )(1-q)(1-q^2 )\cdots (1-q^i)}.$$
Hence, $$(q;q)_{m}\, \phi_{\lambda(i)/(k,k,\cdots,k,0)} = \begin{bmatrix} k\\ i \end{bmatrix}_{q} \begin{bmatrix} m\\ i \end{bmatrix}_{q} 
				(1-q)(1-q^2 ) \dots (1-q^i ).$$
and  
\begin{align*}P_{(m)}(\boldsymbol{x};q,0)P_{(k,\cdots,k,0)}(\boldsymbol{x};q,0)
= \sum\limits_{i=0}^{\min\{m,k\}} \begin{bmatrix} k \\ i \end{bmatrix}_q \begin{bmatrix} m \\ i \end{bmatrix}_q (1-q)\dots (1-q^i )  P_{(k+m-i,k,\dots,k,i)}(\boldsymbol{x};q,0). \end{align*} This completes the proof of (i), in adjunction with \eqref{CL.ch}.
\vspace{.15cm}
    
\noindent (ii). Assume $m \geq k$. To prove (ii), we consider the product $P_{(m)}(\boldsymbol{x};q,0)P_{(k)}(\boldsymbol{x};q,0)$. Using 
\lemref{PRule} and \eqref{P in terms of g}, we have 
$$\begin{array}{l}
P_{(m)}(\boldsymbol{x};q,0)P_{(k)}(\boldsymbol{x};q,0)
= (q;q)_{k}P_{(m)}(\boldsymbol{x};q,0)g_{k}(\boldsymbol{x};q,0) \\
\\
				= (q;q)_{k} \sum\limits_{\mu - (m) = (k)}
\phi_{\mu/ (m)}(q,0) P_{\mu}(\boldsymbol{x};q,0)
				= (q;q)_{k} \sum\limits_{\mu - (m) = (k)}\, \prod\limits_{s \in C_{\mu / (m)}} \dfrac{b_{\mu}(s;q,0)}{b_{(m)}(s;q,0)}\, P_{\lambda}(\boldsymbol{x};q,0) \end{array}$$ 
Now observe that $\mu-(m)=(k)$ if and only if $\mu = (m+k-i,i,0,\cdots)$ for some $0\leq i\leq k$ and for the partition $\mu(i) = (m+k-i,i,0\cdots),$
$$C_{\mu(i)/(m)} = \{(1,m+r), \text{ for } 1\leq r\leq k-i\}\cup \{(r,j), \text{ for } 1\leq j\leq i, 1\leq r\leq 2\}$$
Since, $l_{\mu(i)}(s)>0$, for $s=(1,j)$, $1\leq j\leq i$, 
 $$(a_{\mu(i)}(s), l_{\mu(i)}(s))  = \left\{ \begin{array}{ll} (k-i-r,0) & 
\text{ for } s=(1,m+r),\, 1\leq r\leq k-i\\ (i-j,0) & \text{ for } s=(2,j), \, 
1\leq j\leq i\\  \end{array} \right. $$
and $$(a_{(m)}(s), l_{(m)}(s))  =  (m-j,0) \quad \text{ for } s=(1,j), \, 1\leq j\leq i, $$ substituting the values for $b_{\mu(i)}(s;q,0)$ and  $b_{(m)}(s;q,0)$ we get,
$$\phi_{\mu(i)/(m)} = \dfrac{(1-q^m)(1-q^{m-1})\cdots(1-q^{m-i+1})}{(1-q)(1-q^2)\cdots(1-q^{k-i})(1-q)\cdots(1-q^i)}.$$
Hence $$(q;q)_{k} \phi_{\mu(i)/(m)} = \begin{bmatrix} k\\ i \end{bmatrix}_{q} \begin{bmatrix} m\\ i \end{bmatrix}_{q} 
				(1-q)(1-q^2 ) \dots (1-q^i ), $$
and 
\begin{align*}P_{(m)}(\boldsymbol{x};q,0)P_{(k)}(\boldsymbol{x};q,0)
= \sum\limits_{i=0}^{\min\{m,k\}} \begin{bmatrix} k \\ i \end{bmatrix}_q \begin{bmatrix} m \\ i \end{bmatrix}_q (1-q)\dots (1-q^i )  P_{(m+k-i,i,0,\cdots,0)}(\boldsymbol{x};q,0) \end{align*}
Part (ii) follows using \eqref{CL.ch}.

   \vspace{.15cm}

\noindent (iii). The proof of (iii) is derived from (ii) using automorphism $-w_{0}: \mathfrak{h}_{\mathbb{R}}^{*} \mapsto \mathfrak{h}_{\mathbb{R}}^{*}$ which acts on the dual space of Cartan subalgebra $\mathfrak h$ of $\mathfrak g$, 
sending $\omega_{i} \mapsto \omega_{n+1-i}$. \end{proof}

%
%
%
%
%
%

\section{ Truncated Weyl modules and Filtration of $W_{loc}(m\omega_1)\otimes W_{loc}(k\omega_2)$}

In this section we recall the definition of truncated Weyl modules and establish that for $\mathfrak g=\mathfrak{sl}_3(\mathbb C)$, the tensor product module $W_{loc}(m\omega_1)\otimes W_{loc}(k\omega_2)$ has a filtration by truncated Weyl modules. 

\subsection{} For $n\geq 1$, let $\mathcal A_n = \dfrac{\mathbb C[t]}{(t^n)}$. It was shown in \cite[Lemma 2.2]{CFK} that the truncated current algebra 
 $\mathfrak g\otimes \mathcal A_n$, associated to a  finite-dimensional simple Lie algebra $\mathfrak g$,  is isomorphic to the graded quotient $\dfrac{\mathfrak g\otimes \mathbb C[t]}{\mathfrak g\otimes t^n\mathbb C[t]}$. For $\lambda \in P^+$, the truncated Weyl module $W_{n}(\lambda)$ is defined as the local Weyl module for the truncated current algebra $\mathfrak g\otimes \mathcal A_n$. As $y_\theta$ generates $\mathfrak g$ as a $\bu(\mathfrak g)$-module, it is easy to see that $W_n(\lambda)$ can be naturally considered as the quotient of the local Weyl module $W_{loc}(\lambda)$ as follows : 
 $$W_n(\lambda) \cong W_{loc}(\lambda)/ <(y_{\theta} \otimes t^n)w_\lambda>.$$
 
\vspace{.15cm}


\begin{theorem} Let $\mathfrak g = \mathfrak{sl}_3(\mathbb C)$ and $\lambda=m_1\omega_1+m_2\omega_2$ be a dominant integral weight. Set $|\lambda|=m_1+m_2$, 
$M_{\lambda}=\max\{m_i: i=1,2\}$ and $L_{\lambda}=\min\{m_i: i=1,2\}$. 
\begin{itemize}
\item[i.] For $0\leq j\leq L_{\lambda},$
$$W_{|\lambda|-j}(\lambda) \cong_{\mathfrak{sl}_3[t]} V(\theta)^{\ast j} \ast W_{loc}(\lambda-j \theta).$$

\item[ii.] For $0\leq j < L_{\lambda}$, there exists a short exact sequence, $$0 \longrightarrow \tau_{|\lambda|-j-1} W_{|\lambda-\theta|-j}(\lambda- \theta) \overset{\phi^-}{\longrightarrow} W_{|\lambda|-j}(\lambda) \overset{\phi^+}{\longrightarrow} W_{|\lambda|-j-1}(\lambda) \longrightarrow 0.$$
\end{itemize}
\label{trunc}
\end{theorem}

\noindent We postpone the proof of part(i) of \thmref{trunc} till \secref{M_j(m.n.theta)} and assuming  \thmref{trunc}(i), proceed to prove part(ii) of the theorem. The following notation will be used in the proof.

\noindent {\bf{Notation}} : Given $r\in \mathbb N$ and a sequence with $r$ parts, $\bos =(s_1,s_2,\cdots,s_r)\in \mathbb Z_+^r$, henceforth, we denote the element $(y\otimes t^{s_1}) (y\otimes t^{s_2})\cdots(y\otimes t^{s_r})$ in $\bu(\mathfrak g[t])$ by $y(r,\bos)$.
 
\subsection{Proof of \thmref{trunc}(ii).}  
\begin{lemma} \label{trunc1} Given $\lambda\in P^+$, for $0 \leq j < \min\{\lambda(h_{i}): i =1,2\}$ there exist a surjective $\mathfrak{sl}_3[t]$-module homomorphism $\phi^+: W_{|\lambda|-j}(\lambda) \longrightarrow W_{|\lambda|-j-1}(\lambda)$ such that $$\ker \phi^+= \bu (\mathfrak g[t])(y_\theta \otimes t^{|\lambda|-j-1})w_{\lambda,j}, $$ where we denote by $w_{\lambda,j}$   the image of $w_\lambda$ in the truncated Weyl module $W_{|\lambda|-j}(\lambda)$.
\end{lemma}
\proof From the definition of the truncated Weyl modules it is clear that there exists a  $\mathfrak g[t]$-module homomorphism $\phi^+: W_{|\lambda|-j}(\lambda) \longrightarrow W_{|\lambda|-j-1}(\lambda)$ such that 
$$\phi^+(w_{\lambda,j}) = w_{\lambda,j-1} \qquad \text{ and }\qquad \ker \phi^+= \bu(\mathfrak g[t])(y_\theta \otimes t^{|\lambda|-j-1})w_{\lambda,j}.$$


\begin{lemma}\label{trun2}
 For $0 \leq j < \min\{\lambda(h_i): i=1,2\}$, there exist a $\mathfrak{sl}_3[t]$-module isomorphism $\phi^-: W_{|\lambda-\theta|-j}(\lambda- \theta) \longrightarrow \ker\phi^+$, where 
$\ker \phi^+= U(\mathfrak g[t])(y_\theta \otimes t^{|\lambda|-j-1})w_{\lambda,j}$.
\end{lemma}
\proof By the defining relations of $W_{loc}(\lambda)$ we know, $y_{\alpha}\otimes t^s w_{\lambda} =0$ whenever $s\geq \lambda(h_\alpha).$ As $W_{|\lambda|-j}(\lambda)$ is a quotient of $W_{loc}(\lambda)$ and   $|\lambda|-j =\lambda(h_1)+\lambda(h_2)-j >\lambda(h_i)$ for $j<\lambda(h_i)$, we see that as $j< \min \{\lambda(h_i): i=1,2\}$,
$$x_{ii}\otimes t^r (y_\theta \otimes t^{|\lambda|-j-1})w_{\lambda,j} =  y_{\alpha_1+\alpha_2-\alpha_{i}}\otimes t^{|\lambda|-j-1+r}w_{\lambda,j}=0, \qquad \text{ for } i=1,2. $$
Further, using \eqref{loc.Weyl.2.5}, we have, $$ x_{\theta}\otimes t^r (y_\theta \otimes t^{|\lambda|-j-1})w_{\lambda,j} =  h_\theta\otimes t^{|\lambda|-j-1+r}w_{\lambda,j}=0, $$ and owing to the defining relations of the truncated Weyl module $W_{|\lambda|-j}(\lambda)$ we have,
$$h\otimes t^s( (y_\theta \otimes t^{|\lambda|-j-1})w_{\lambda,j}) = (\lambda-\theta)(h)\, \delta_{s,0} (y_\theta \otimes t^{|\lambda|-j-1})w_{\lambda,j}, \qquad \text{ for } s\in \mathbb Z_+.$$ From the universal property of local Weyl modules it follows that $\bu (y_\theta \otimes t^{|\lambda|-j-1})w_{\lambda,j}$ is a quotient of $W_{loc}(\lambda-\theta).$

Further, as  $j < \min\{\lambda(h_i):i=1,2\}$,    
$2(j+1)\leq \lambda(h_1)+\lambda(h_2)=|\lambda| $, implying 
$2|\lambda|-2j-1= |\lambda|+1 +(|\lambda|-2j-2) \geq |\lambda|+1.$ Hence by defining relation of $W_{|\lambda|-j}(\lambda)$, $(y_{\theta} \otimes 1)^{2|\lambda|-2j-1} w_{\lambda,j} =0$. Consequently,
     $$\begin{array}{ll}
            (x_{\theta} \otimes t)^{2|\lambda|-2j-3}(y_{\theta} \otimes 1)^{2|\lambda|-2j-1} w_{\lambda,j} =0,\\
         (y_{\theta} \otimes t^{|\lambda|-j-1})(y_{\theta} \otimes t^{|\lambda|-j-2}) w_{\lambda,j}+ (y_{\theta} \otimes t^{|\lambda|-j-3})(y_{\theta} \otimes t^{|\lambda|-j})w_{\lambda,j} =0\\
        (y_{\theta} \otimes t^{|\lambda|-j-2}) (y_{\theta} \otimes t^{|\lambda|-j-1}) w_{\lambda,j}  =0
     \end{array}$$
    Therefore it follows that $\bu(\mathfrak g[t])(y_{\theta} \otimes t^{|\lambda|-j-1})w_{\lambda,j}$ is a quotient of  $W_{|\lambda|-j-2}(\lambda-\theta)$. Now, using  \thmref{trunc}(i), we see on one hand, 
\begin{equation}
\begin{array}{ll}\dim \bu(\mathfrak g[t])(y_{\theta} \otimes t^{|\lambda|-j-1})w_{\lambda,j} &\leq \dim W_{|\lambda|-j-2}(\lambda-\theta) = \dim W_{|\lambda-\theta|-j}(\lambda-\theta)\\ 
& \leq (\dim V(\theta))^j\dim W_{loc}(\lambda-(j+1)\theta) = 8^j 3^{|\lambda|-2(j+1)}.\end{array} \label{dim.1}\end{equation}
On the other hand as $\bu(\mathfrak g[t])(y_{\theta} \otimes t^{|\lambda|-j-1})w_{\lambda,j} = \ker \phi^+$, it follows that $$\dim \bu(\mathfrak g[t])(y_{\theta} \otimes t^{|\lambda|-j-1})w_{|\lambda|,j} = \dim W_{|\lambda|-j}(\lambda)-\dim W_{|\lambda|-j-1}(\lambda).$$ Hence by \thmref{trunc}(i) it follows that ,
\begin{equation}\dim \bu(\mathfrak g[t])(y_{\theta} \otimes t^{|\lambda|-j-1})w_{\lambda,j} = 8^j 3^{|\lambda|-2j}-8^{j+1} 3^{|\lambda|-2(j+1)}= 8^j  3^{|\lambda|-2(j+1)}.\label{dim.2}\end{equation} From \eqref{dim.1} and \eqref{dim.2} we thus conclude that  $\ker\phi^- = \bu(\mathfrak g[t])(y_{\theta} \otimes t^{|\lambda|-j-1})w_{\lambda,j}$ is isomorphic to $W_{|\lambda-\theta|-j}(\lambda-\theta)$ as a $\mathfrak{sl}_3[t]$-module. 
\endproof

\noindent \thmref{trunc}(ii) is now an immediate consequence of \lemref{trunc1} and \lemref{trun2}.

\subsection{} Using \thmref{trunc}(ii) we now obtain the graded character of truncated Weyl modules of $\mathfrak{sl}_3[t]$ in terms of graded characters of local Weyl modules.
  
\begin{lemma} \label{t in local} For $\lambda\in P^+$, let $L_\lambda$ be as defined in \thmref{trunc}. For $0\leq j \leq L_{\lambda}$,
$$\ch_{gr} W_{|\lambda|-j}(\lambda)  = \sum\limits_{i=0}^{j} (-1)^{i} \begin{bmatrix} j \\ i \end{bmatrix}_{q} q^{i(|\lambda|-j) - i(i-1)/2} \ch_{gr} W_{loc}(\lambda-i\theta).$$ 
\end{lemma}
\begin{proof} We prove the lemma by induction on $j$.\\
From the short exact sequence in \thmref{trunc}(ii) for $j=1$, 
$$0 \longrightarrow \tau_{|\lambda|-1}W_{loc}(\lambda-\theta) \longrightarrow W_{loc}(\lambda) \longrightarrow W_{|\lambda|-1}(\lambda)\longrightarrow 0$$
we get $\ch_{gr} W_{|\lambda|-1}(\lambda) = \ch_{gr}W_{loc}(\lambda) - q^{|\lambda|-1}\ch_{gr} W_{loc}(\lambda-\theta) ,$ which shows that the lemma holds for $j=1$.
Assume that we have proved the statement for all $p<j$ and for all dominant integral weights $\mu<\lambda$.  Using  the following short exact sequence from \thmref{trunc}(ii), 
$$0 \longrightarrow \tau_{|\lambda|-j} W_{|\lambda|-j-1}(\lambda-\theta) \longrightarrow W_{|\lambda|-j+1}(\lambda) \longrightarrow W_{|\lambda|-j}(\lambda) 
\longrightarrow 0$$ we get $$\ch_{gr} W_{|\lambda|-j}(\lambda) = 
\ch_{gr} W_{|\lambda|-(j-1)}(\lambda) -q^{|\lambda|-j} 
\ch_{gr} W_{|\lambda-\theta|-(j-1)}(\lambda-\theta). $$
Now using the induction hypothesis and the graded character for $W_{|\lambda-\theta|-j-1}(\lambda-\theta)$ we get,
\begin{align*}
\ch_{gr} W_{|\lambda|-j}(\lambda)&
 = \sum\limits_{i=0}^{j-1} (-1)^{i} \begin{bmatrix} j-1 \\ i \end{bmatrix}_{q} q^{i(|\lambda|-j+1) - \frac{i(i-1)}{2}} \ch_{gr} W_{loc}(\lambda-i\theta)\\& -q^{|\lambda|-j} \sum\limits_{i=0}^{j-1} (-1)^{i} \begin{bmatrix} j-1 \\ i \end{bmatrix}_{q} q^{i(|\lambda|-j-1) - \frac{i(i-1)}{2}} \ch_{gr} W_{loc}(\lambda-\theta-i\theta) \\
& = \sum\limits_{i=0}^{j-1} (-1)^{i} \begin{bmatrix} j-1 \\ i \end{bmatrix}_{q} q^{i(|\lambda|-j+1) - \frac{i(i-1)}{2}} \ch_{gr} W_{loc}(\lambda-i\theta)\\
& + q^{|\lambda|-j}\sum\limits_{i=1}^{j} (-1)^{i} \begin{bmatrix} j-1 \\ i-1 \end{bmatrix}_{q} q^{(i-1)(|\lambda|-j-1) - \frac{(i-2)(i-1)}{2}} \ch_{gr} W_{loc}(\lambda-\theta) \\
& = \ch_{gr} W_{loc}(\lambda) + (-1)^j q^{j(|\lambda|-j) - \frac{j(j-1)}{2}}\ch_{gr} W_{loc}(\lambda-j\theta)\\
& + \sum\limits_{i=1}^{j-1} (-1)^{i} (\begin{bmatrix} j-1 \\ i \end{bmatrix}_{q} q^i + \begin{bmatrix} j-1 \\ i-1 \end{bmatrix}_{q} ) q^{i(|\lambda|-j) - \frac{i(i-1)}{2}} \ch_{gr} W_{loc}(\lambda-i\theta)\\
& = \sum\limits_{i=0}^{j} (-1)^{i} \begin{bmatrix} j \\ i \end{bmatrix}_{q} q^{i(|\lambda|-j) - \frac{i(i-1)}{2}} \ch_{gr} W_{loc}(\lambda-i\theta).
\end{align*} This proves the lemma. \label{trunc.ch} \end{proof}

\subsection{} Recall the following polynomial identity in $x, q$ which is equivalent to the q-binomial theorem:
\begin{equation}  \sum_{r=0}^{j} (-1)^{j-r} \begin{bmatrix} j \\ r \end{bmatrix}_{q} q^{{ j-r \choose 2}} x^{r}  = (x-1) (x-q) \cdots (x-q^{j-1}) \label{q-binomial}
\end{equation} We shall now use it to obtain the graded character of the $\mathfrak{sl}_3[t]$-module  $W_{loc}(m\omega_{1}) \otimes W_{loc}(k\omega_{2})$ in terms of the graded character of truncated Weyl modules.

\begin{lemma}\label{t product} For $m,k\in \mathbb N$, let $M=\max\{m,k\}$ and $L=\min\{m,k\}$.
Then, \begin{align}\label{t.product}
\ch_{gr} W_{loc}(m\omega_{1}) \otimes W_{loc}(k\omega_{2})  = 
\sum\limits_{i=0}^{L} \begin{bmatrix} L \\ i \end{bmatrix}_{q} \ch_{gr} W_{M-i}(m\omega_1+k\omega_2-i\theta).\end{align}
\end{lemma}
\proof From \lemref{trunc.ch},  we see that 
for $m',k'\in\mathbb N$ with $L'=\min\{m',k'\}$ we have
$$\ch_{gr} W_{m'+k'-L'}(m'\omega_{1} + k' \omega_{2}) = \sum\limits_{i=0}^{L'} (-1)^{i} \begin{bmatrix} L' \\ i \end{bmatrix}_{q} q^{i(m'+k'-L') - i(i-1)/2} \ch_{gr} W_{loc}(m'\omega_{1} + k'\omega_{2}-i\theta).$$ In particular 
for each $0\leq i\leq L$ we have
\begin{equation}
\begin{array}{ll}\ch_{gr} & W_{M-i}(m\omega_1+k\omega_2-i\theta) = \ch_{gr} W_{M-i}((m-i)\omega_{1} + (k-i) \omega_{2}) \\ &= \sum\limits_{r=0}^{L-i} (-1)^{r} \begin{bmatrix} L-i \\ r \end{bmatrix}_{q} q^{r(M-i) - r(r-1)/2} \ch_{gr} W_{loc}((m-i)\omega_{1} + (k-i)\omega_{2}-r\theta).\label{m-i.trunc}\end{array}\end{equation}
Using \eqref{m-i.trunc}, the right hand side of \eqref{t.product} can be written as :
\begin{equation}\begin{array}{ll} \sum\limits_{i=0}^{L} \begin{bmatrix} L \\ i \end{bmatrix}_{q} \ch_{gr} W_{M-i}(m\omega_1+k\omega_2-i\theta) =  \sum\limits_{i=0}^{L} \begin{bmatrix} L \\ L-i \end{bmatrix}_{q} \ch_{gr} W_{M-i}(m\omega_1+k\omega_2-i\theta) \\= \sum\limits_{i=0}^{L} \begin{bmatrix} L \\ L-i \end{bmatrix}_{q} \sum\limits_{r=0}^{L-i} (-1)^{r} \begin{bmatrix} L-i \\ r \end{bmatrix}_{q} q^{r(M-i) - r(r-1)/2} \ch_{gr} W_{loc}(m\omega_{1} + k\omega_{2}-(r+i)\theta).
\end{array}\label{subs.}\end{equation}
Note that the coefficient of $\ch_{gr} W_{loc}(m\omega_1+k\omega_2-j\theta)$ in \eqref{subs.} is equal to:
\begin{equation} \label{subs.1}
\sum\limits_{r=0}^j (-1)^{r} \begin{bmatrix} L \\ L-(j-r) \end{bmatrix}_{q}  \begin{bmatrix} L-(j-r) \\ r \end{bmatrix}_{q} q^{r(M-j-r)-\frac{r(r-1)}{2}} 
= \begin{bmatrix} L \\ j \end{bmatrix}_{q}  \sum\limits_{r=0}^j (-1)^{r} \begin{bmatrix} j \\ r\end{bmatrix}_{q}  q^{r(M-j-r)-\frac{r(r-1)}{2}} \end{equation}
On the other hand putting $x=q^M $ in \eqref{q-binomial} we get
 $$\sum_{r=0}^{j} (-1)^{j-r} \begin{bmatrix} j \\ r \end{bmatrix}_{q} q^{r M + { j-r \choose 2}}  = (-1)^{j} q^{\frac{j(j-1)}{2}}(1-q^M )  \cdots (1-q^{M-j+1} ),$$
which in turn gives, \begin{equation} \label{subs.2}
\begin{array}{ll}
\sum\limits_{r=0}^{j} (-1)^r \begin{bmatrix} j \\ r \end{bmatrix}_{q} q^{r(M-j+r) - \frac{r(r-1)}{2}}& = (1-q^M ) (1-q^{M-1}) \cdots (1-q^{M-j+1})\\ &= \begin{bmatrix} M \\ j \end{bmatrix}_{q} (1-q)(1-q^2) \cdots (1-q^j ).\end{array}\end{equation}
Substituting the coefficient of $W_{loc}(m\omega_1+k\omega_2-j\theta)$ in \eqref{subs.} using \eqref{subs.1} and \eqref{subs.2} we thus see that the right hand side of \eqref{t.product} is equal to :
$$ \sum\limits_{i=0}^L \begin{bmatrix} M \\ i \end{bmatrix}_{q} \begin{bmatrix} L \\ i \end{bmatrix}_{q}(1-q)\cdots (1-q^{i}) \ch_{gr} W_{loc}(m\omega_{1} + k\omega_2-i\theta),$$ which by \lemref{loc in loc} is equal to the  graded character of
$W_{loc}(m\omega_1)\otimes W_{loc}(k\omega_2)$. This completes the proof of the lemma.  \endproof

\vspace{.15cm}

\subsection{Filtration of $W_{loc}(m\omega_1)\otimes W_{loc}(k\omega_2)$} We now prove the main result of this section.
\begin{theorem} \label{filter.1.2}For $m,k\in \mathbb N$, let $M =\max\{m,k\}$ and $L=\min\{m,k\}$. The $\mathfrak{sl}_3[t]$-module $W_{loc}(m\omega_{1}) \otimes W_{loc}(k\omega_{2})$ has a filtration by $\mathfrak{sl}_3[t]$-submodules such that the successive quotients are isomorphic to truncated Weyl modules: 
    \begin{align*} 
    &\tau_{l_r}W_{M-r}((m-r)\omega_{1} + (k-r) \omega_{2}), \qquad \text{ for } 0 \leq r \leq L,\quad  0\leq l_r\leq (L-r)r. 
    \end{align*}
\end{theorem}
\proof Given $m, k \in \mathbb{N}$, suppose that $k \leq m$. Let $w_1 $ and $w_2 $ be the generators of the local Weyl modules $W_{loc}(m\omega_{1})$ and $ W_{loc}(k\omega_2 )$ respectively. By \propref{POP.basis}, $\mathfrak B(k\omega_2)$ is an  ordered basis of $W_{loc}(k\omega_2)$ which is  parametrized by $\mathfrak P(k\omega_2)$, the set of  partition overlaid Gelfand-Tsetlin patterns with bounding sequence $(k,k,0)$. For each $(\eta,\pi_\eta)\in \mathfrak P(k\omega_2)$, set $$v_{\eta,\pi_\eta} = \omega_1\otimes y(\eta,\pi_\eta)\omega_2,$$ 
$$\mathcal B_{m,k}^{(\eta,\pi_\eta)\leq } = \{v_{\zeta,\pi_\zeta} : (\zeta,\pi_\zeta)\in \mathfrak P(k\omega_2), \, (\eta,\pi_\eta)\leq (\zeta,\pi_\zeta) \},$$ and let $\mathcal T(\eta,\pi_\eta)$ be the submodule of $W_{loc}(m\omega_1)\otimes W_{loc}(k\omega_2)$ generated by $\mathcal B_{m,k}^{(\eta,\pi_\eta)\leq }.$ Let $(\eta_0,\pi_{\eta_0})\in \mathfrak P(k\omega_2)$ be the element that  corresponds to the lowest weight vector of $W_{loc}(k\omega_2)$. Then by \lemref{generator of tensor}, $$W_{loc}(m\omega_1)\otimes W_{loc}(k\omega_2) = \mathcal T(\eta_0,\pi_{\eta_0}).$$ Further, by definition, 
$\mathcal T(\zeta,\pi_\zeta)$ is a submodule of  $\mathcal T(\eta,\pi_\eta)$ for all $(\zeta,\pi_\zeta)\in \mathfrak P(k\omega_2)$ such that $(\eta,\pi_\eta) \leq (\zeta,\pi_\zeta)$. This gives a descending chain of submodules of the $\mathfrak{sl}_3[t]$-module $W_{loc}(m\omega_1)\otimes W_{loc}(k\omega_2)$. Consequently, $$\ch_{gr} W_{loc}(m\omega_1)\otimes W_{loc}(k\omega_2) = \sum\limits_{(\eta,\pi_\eta)\in \mathfrak P(k\omega_2)} \ch_{gr} \dfrac{\mathcal T(\eta,\pi_\eta)}{ \mathcal T(\eta^+,\pi_\eta^+)}, $$ where $(\eta^+,\pi_{\eta^+})\in \mathfrak P(k\omega_2)$ is as defined in \secref{POP}. 

 Given $p,n,k\in \mathbb Z_+$ let $\eta_k(p,q)$ be a triple of arrays given by 
 $\eta_k(p,q) :=  \begin{array}{ccccc} &k-q&\\ &k\qquad  \, k-p\\ k & k & 0  \end{array}.$ Then, $(\eta,\pi_\eta)\in \mathfrak{P}(k\omega_2)$ whenever $\eta=\eta_k(p,q)$ with $q\leq p\leq k.$\\ We claim that $ {\mathcal T(\eta,\pi_\eta)}\subseteq { \mathcal T(\eta^+,\pi_\eta^+)}$ for all $\eta=\eta_k(p,q)$ with $p\neq q$.

Let $\mathcal T(\varphi,\pi_{\varphi})$ be the submodule of $W_{loc}(m\omega_1)\otimes W_{loc}(k\omega_2)$ generated by $w_1\otimes w_2.$ By definition $\mathcal T(\varphi,\pi_{\varphi}) \subset \mathcal T(\eta,\pi_\eta)$ for all $(\eta,\pi_\eta)\in \mathfrak P(k\omega_2)$.
As $y_{22}\otimes t^s.w_1 =0,$ given a positive integer $p$ and a $p$-tuple of integers $\bos$,
$$w_1\otimes y_{22}(p,\bos)w_2 = y_{22}(p,\bos)(w_1\otimes w_2)\in \mathcal T(\varphi,\pi_{\varphi}). $$ This proves $\dfrac{\mathcal T(\nu,\pi_\nu)}{ \mathcal T(\nu^+,\pi_\nu^+)} = 0, $ for all $(\nu,\pi_{\nu})\in \mathfrak P(k\omega_2)$ with 
$\nu =\eta_k(p,0)$.

\noindent Likewise, using the fact that $y_{22}\otimes t^s.w_1 =0$ and $y_{11}\otimes t^s.w_2 =0$, $\forall \, s \in \mathbb Z_+$, for any $p$-tuple of integers $\bos$, we have,
\begin{equation} \label{1} 
y_{11}^{(p)} y_{22}(p,\bos).w_2=y_{12}(p,\bos)w_2, \quad \text{ whenever } p\leq k. \end{equation}
Given $\bos=(s_1,s_2,\cdots,s_{p+1})$, 
let $\hat{\bos}_i:=(s_1,\cdots,s_{i-1},s_{i+1},\cdots,s_{p+1})$, $1\leq i\leq p+1$. Then using \eqref{1} we have, $y_{11}^{(p)}y_{22}(p+1,\bos)w_2=\sum\limits_{i=1}^{p+1} y_{22}\otimes t^{s_i}y_{12}(p,\hat{\bos}_i)w_2 $ and consequently,  
$$\begin{array}{ll}\sum\limits_{i=1}^{p+1}y_{22}\otimes t^{s_i}(w_1\otimes y_{12}(p,\hat{\bos}_i)w_2) &= \sum\limits_{i=1}^{p+1} w_1\otimes (y_{22}\otimes t^{s_i})y_{12}(p,\hat{\bos}_i)w_2\\&=w_1 \otimes(  y_{11}^{(p)})(y_{22}(p+1,\bos)w_{2}).\end{array} $$ Given $\bos=(s_1,s_2,\cdots,s_p),$ note that 
$y_{11}\otimes t^r y_{22}(p,\bos)w_2 = \sum\limits_{i=1}^p y_{22}(p-1,\hat{\bos}_i)y_{12}(1,s_i+r) w_2.$ Thus for $\nu =\eta_k(p,q)$ with $p>q$, 
using the above calculations suitably many times we see that 
$\mathcal T(\nu,\pi_\nu)\subset \mathcal T(\nu^+,\pi_\nu^+)$.   
Hence it follows that
$$\ch_{gr} W_{loc}(m\omega_1)\otimes W_{loc}(k\omega_2) = \sum\limits_{0\leq p\leq k}
\ch_{gr} \dfrac{\mathcal T(\eta(p),\pi_{\eta(p)})}{ \mathcal T(\eta(p)^+,\pi_{\eta(p)}^+)}, $$ where 
$\eta(p) = \eta_k(p,p)$.

To complete the proof of the theorem we now consider the modules $\dfrac{\mathcal T( \eta(p),\pi_{\eta(p)})}{ \mathcal T(\eta(p)^+,\pi_{\eta(p)}^+)}$, for $0\leq p\leq k$. Since  
\begin{align*}&[(x_{\alpha_i}\otimes t^r), y_{12}(p,\bos)]w_2 = 
\sum\limits_{i=1}^{p}  (y_{\theta-\alpha_i} \otimes t^{s_i+r}) y_\theta(p-1, \hat{\bos}_i) w_2,    \quad  \text{for } i=1,2\\ 
&[(x_{12}\otimes t^r), y_{12}(p,\bos)]w_2 = 
\sum\limits_{0\leq i\leq p } y_{12}(i-1, \bos(i))h_\theta\otimes t^{s_i+r} y_{12}(p-i,\bos'(i))w_2\end{align*} where for $\bos=(s_1,\cdots,s_p)$, 
 $\bos(i) := (s_1,\cdots,s_{i-1})$ and $\bos'(i):=(s_{i+1},\cdots,s_p)$ for $0\leq i\leq p$; and $$[(h\otimes t^r), y_{12}(p,\bos)]w_{k\omega_2} = \sum\limits_{i=1}^{p} A_i (y_{\theta} \otimes t^{s_i+r}) y_\theta(p-1, \hat{\bos}_i) w_{k\omega_2}, $$ with  $A_i\in \mathbb C, \ 1\leq i\leq p$, using \eqref{1}, we see that, 
$$\begin{array}{l} x_\alpha \otimes t^r v_{\eta(p),\pi_{\eta(p)}} \in \mathcal T(\eta(p)^+, \pi^+_{\eta(p)}), \qquad \forall \alpha\in \Phi^+, \\ 
h \otimes t^r v_{\eta(p),\pi_{\eta(p)}} \in \mathcal T(\eta(p)^+, \pi^+_{\eta(p)}), \qquad h\in \mathfrak h.\end{array}$$ This shows that $\dfrac{\mathcal T(\eta(p),\pi_{\eta(p)})}{ \mathcal T(\eta(p)^+,\pi_{\eta(p)}^+)}$ is a highest weight module generated by the image of $v_{\eta(p),\pi_{\eta(p)}}$  under the natural projection from $T(\eta(p), \pi_{\eta(p)})$ onto $\dfrac{\mathcal T(\eta(p),\pi_{\eta(p)})}{ \mathcal T(\eta(p)^+,\pi_{\eta(p)}^+)}$ for each $0\leq p\leq k$.
Further, as the $\mathfrak h$-weight of $v_{\eta(p),\pi_{\eta(p)}}$ is $m\omega_1+k\omega_2-p\theta$ and $$(y_{\alpha}\otimes t^m )v_{\eta(p),\pi_{\eta(p)}}=0, \qquad \forall\, \alpha\in \Phi^+,$$ we see that $\mathcal T(\varphi,\pi_\varphi)$ is isomorphic to the truncated local Weyl module $W_m(m\omega_1+k\omega_2)$ and for all $(\eta(p),\pi_{\eta(p)})\in \mathfrak P(k\omega_2)$ with  $1\leq p\leq k$,  $\dfrac{\mathcal T(\eta(p),\pi_{\eta(p)})}{ \mathcal T(\eta(p)^+,\pi_{\eta(p)}^+)}$ is a quotient of the truncated local Weyl module $\tau_{|(\eta(p),\pi_{\eta(p)})|_{22}}(W_{m}((m-r)\omega_1 + (k-r)\omega_2 ))$, where $| . |_{22}$ is the function defined in \secref{POP}. 

Setting  the $\mathbb Z$-grade of $y_{12}(p,\pi_{\eta(p)})$ by $r(\eta(p))$, we thus see that, 
\begin{align} \label{ch.12.2}
\ch_{gr} W_{loc}(m\omega_1)\otimes W_{loc}(k\omega_2) = &\sum\limits_{j=1}^k \sum\limits_{(\eta(p),\pi_{\eta(p)})\in \mathfrak P(k\omega_2)} q^{r(\eta(p))} \ch_{gr} \dfrac{\mathcal T(\eta(p),\pi_{\eta(p)})}{ \mathcal T(\eta(p)^+,\pi_{\eta(p)}^+)}\\ &+ \ch_{gr} W_m(m\omega_1+k\omega_2) \end{align}
Notice that by definition of $\mathfrak P(k\omega_2)$, for a fixed $r\in \mathbb N,$ there exists a one-to-one correspondence between the set $\{(\eta(p), \pi_{\eta(p)})\in \mathfrak P(k\omega_2) : |(\eta(p), \pi_{\eta(p)})|_{22}=r \}$ and the set $\{(s_1\leq \cdots\leq s_p)\in \mathbb Z^k: 0\leq s_i\leq k-p,\, \sum\limits_{i=1}^p s_i=r \}$. So from standard  q-binomial theory it follows that 
$$\# \{(\eta(p), \pi_{\eta(p)})\in \mathfrak P(k\omega_2) : |(\eta(p), \pi_{\eta(p)})|_{22}=r \} = \text{coefficient of } q^r \text{ in } \begin{bmatrix} k \\ p \end{bmatrix}_{q}. $$
Now, comparing \eqref{ch.12.2} with the graded character \eqref{t.product} of $W_{loc}(m\omega_1)\otimes W_{loc}(k\omega_2)$  and observing that $\dfrac{\mathcal T(\eta(p),\pi_{\eta(p)})}{ \mathcal T(\eta(p)^+,\pi_{\eta(p)}^+)}$ is a quotient of $W_m((m-p)\omega_1+(k-p)\omega_2)$, we conclude that 
 $\dfrac{\mathcal T(\eta(p),\pi_{\eta(p)})}{ \mathcal T(\eta(p)^+,\pi_{\eta(p)}^+)} $, $1\leq p\leq k$, is isomorphic to $W_{m-p}((m-p)\omega_1+(k-p)\omega_2)$ as $\mathfrak{sl}_3[t]$-module. Since  by the characterization of $\mathfrak P(k\omega_2)$,  $|(\eta(p), \pi_{\eta(p)})|_{22}$ ranges between $0$ and $(k-p)p$ and the theorem follows in the case when $m\geq k$.\endproof

%
%
%
%

\section{ The module $M(\nu,\lambda)$ and Filtration of $W_{loc}(m\omega_i)\otimes W_{loc}(k\omega_i)$ for $i=1,n$}

For a pair of dominant integral weights $(\nu,\lambda)$, the $\mathfrak{sl}_{n+1}[t]$-modules $M(\nu, \lambda)$ were introduced in \cite{MR4224104}. In this section we recall their definition and show that the tensor product modules $W_{loc}(m\omega_i)\otimes W_{loc}(k\omega_i)$, for $i=1,n$ have a filtration by graded $\mathfrak g[t]$-modules where the successive quoients are isomorphic to $M(\nu,\lambda)$-modules. 

\subsection{} 
The following definition of the $\mathfrak{sl}_{n+1}[t]$-module $M(\nu,\lambda)$ was given in \cite{MR4224104}. \begin{definition}
    For $\lambda \in P^+$, suppose that $\lambda= 2\lambda_0+\lambda_1$, where $\lambda_0, \lambda_1\in P^+$ are such that $\lambda_1(h_{\alpha_i})\leq 1$ for $i\in I$. Then for a pair of dominant  integral weights $(\nu,\lambda)$, the $\mathfrak{sl}_{n+1}[t]$-module $M(\nu,\lambda)$ is defined to be quotient of $W_{loc}(\lambda+\nu)$ by the submodule generated by the elements:
   $$ \begin{array}{cc}
       (y_{\alpha} \otimes t^{\lceil{\lambda(h_{\alpha})/2}\rceil+\nu(h_{\alpha})} ) w_{\nu+\lambda}, \quad \forall \, h \in \mathfrak{h}, \alpha \in \Phi^+.   
    \end{array}$$ 
\end{definition}

\noindent In this context the following was proved in \cite{MR4224104}.

\begin{proposition}\cite{MR4224104}\label{M.mu.nu} Let $\mathfrak g =  \mathfrak{sl}_{n+1}(\mathbb C)$ and $\nu, \lambda \in P^+ $ be such that $\nu(h_j) \geq 2$ for some $1 \leq j \leq n.$ 
\begin{itemize} \item[i.] As $\mathfrak{sl}_{n+1}[t]$-modules, 
 $M(\nu, \lambda)$ is isomorphic to $D(2, \lambda) \ast W_{loc}(\nu).$ \label{M.m.n.thm}
\item[ii.] For $\nu=\nu_1\omega_1+\nu_2\omega_2$,  $\nu'=\nu_n\omega_n+\nu_{n-1}\omega_{n-1}$ and $\lambda = 2\lambda_1\omega_1$, $\lambda'=2\lambda_n\omega_n$, 
 \label{truc in local}
$$\ch_{gr} M(\nu,\lambda)= \sum\limits_{i=0}^{\lambda_1} (-1)^{i} \begin{bmatrix} \lambda_1 \\ i \end{bmatrix}_{q} q^{i(\lambda_1+\nu_1) - i(i-1)/2} \ch_{gr} W_{loc}((2\lambda_1+\nu_1-2i)\omega_{1} + (\nu_2+i)\omega_{2})$$
$$\ch_{gr} M(\nu', \lambda')  =  \sum\limits_{i=0}^{\lambda_n} (-1)^{i} \begin{bmatrix} \lambda_n \\ i \end{bmatrix}_{q} q^{i(\lambda_n+\nu_n) - i(i-1)/2} \ch_{gr} W_{loc}((\nu_{n-1}+i)\omega_{n-1} + (2\lambda_n+\nu_n-2i)\omega_{n}).$$ 
\end{itemize}
\end{proposition}
\noindent \begin{proof} While (i) was proved in \cite[Corollary 2.8]{MR4224104}, part(ii) can be deduced using Equations (2.8),(2.9) and (3.6) from \cite{MR4224104}. \end{proof}

\vspace{.15cm}

\subsection{}We now give the graded character of the $\mathfrak{sl}_{n+1}[t]$-modules  $W_{loc}(m \omega_{i}) \otimes W_{loc}(k\omega_{i})$ for $i=1,n$, in terms of the graded characters of $M(\nu,\lambda)$-modules.
\begin{lemma} \label{tensor product}
Given $m,k \in \mathbb{Z}_{\geq 0}$ , let $\mathcal M =\max\{m,k\}$ and $L=\min\{m,k\}$. Then,
\begin{align}
\ch_{gr} W_{loc}(m\omega_{1}) \otimes W_{loc}(k\omega_{1}) & = \sum\limits_{i=0}^{L} \begin{bmatrix} L \\ i \end{bmatrix}_{q} \ch_{gr} M(({\mathcal M} -L)\omega_{1}+ i \omega_{2}, 2(L-i)\omega_{1}) \label{ch.M.n.l}\\
\ch_{gr} W_{loc}(m\omega_{n}) \otimes W_{loc}(k\omega_{n}) & = \sum\limits_{i=0}^{L} 
\begin{bmatrix} L \\ i \end{bmatrix}_{q} \ch_{gr} M(i \omega_{n-1}+ ({\mathcal M}-L) \omega_{n}, 2(L-i)\omega_{n})  \label{ch.M.n.l.2}
\end{align}
\end{lemma}
\proof The proof of \lemref{tensor product} is similar to the proof of \lemref{t product}.
Using \propref{truc in local}(ii), the right hand side of \eqref{ch.M.n.l} can be written as :\\
\begin{equation}\label{omega_1} \sum_{i=0}^{L} \sum_{r=0}^{L-i} (-1)^r  \begin{bmatrix} L \\ i \end{bmatrix}_{q} \begin{bmatrix} L-i \\ r \end{bmatrix}_{q}q^{r({\mathcal M}-i) - \frac{r(r-1)}{2}} \ch_{gr} W_{loc}(({\mathcal M}+L-2i-2r)\omega_{1} + (r+i)\omega_{2}) \end{equation}
Note that the coefficient of $\ch_{gr} W_{loc}(({\mathcal M}+L-2j)\omega_1+j\omega_2)$ in the above equation is : \\
  $$ \begin{bmatrix} L \\ j \end{bmatrix}_{q} (\sum\limits_{r=0}^{j} (-1)^r  \begin{bmatrix} j \\ r \end{bmatrix}_{q}q^{r(m-j+r) - \frac{r(r-1)}{2}}). $$
Now, as in the proof of \lemref{t product}, putting $x=q^{\mathcal M}$  in \eqref{q-binomial} we get, 
\begin{equation}\label{omega_1.2}  \begin{array}{ll}  \begin{bmatrix} L \\ j \end{bmatrix}_{q}
\sum\limits_{r=0}^{j} (-1)^r \begin{bmatrix} j \\ r \end{bmatrix}_{q} q^{r({\mathcal M}-j+r) - \frac{r(r-1)}{2}}& = \begin{bmatrix} L \\ j \end{bmatrix}_{q} (1-q^{\mathcal M} ) (1-q^{\mathcal M-1}) \cdots (1-q^{{\mathcal M}-j+1})\\ &= \begin{bmatrix} L \\ j \end{bmatrix}_{q}\begin{bmatrix} \mathcal M \\ j \end{bmatrix}_{q} (1-q)(1-q^2) \cdots (1-q^j ).\end{array}\end{equation}
Substituting the coefficient of $\ch_{gr} W_{loc}(({\mathcal M}+L-2j)\omega_1+j\omega_2)$ in \eqref{omega_1} and using \eqref{omega_1.2} we see that the right hand side of \eqref{ch.M.n.l} is equal to :
$$ \sum\limits_{j=0}^L \begin{bmatrix} \mathcal M \\ j \end{bmatrix}_{q} \begin{bmatrix} L \\ j \end{bmatrix}_{q}(1-q)\cdots (1-q^{j}) \ch_{gr}  W_{loc}(({\mathcal M}+L-2j)\omega_1+j\omega_2),$$ which,  by \lemref{loc in loc}(ii), is equal to the  graded character of
$W_{loc}(m\omega_1)\otimes W_{loc}(k\omega_1)$. This completes the proof of \eqref{ch.M.n.l}. The graded character \eqref{ch.M.n.l.2} of $W_{loc}(m\omega_n)\otimes W_{loc}(k\omega_n)$ in terms of $\ch_{gr} M(\nu,\lambda)$ can be proved in the same way using \propref{truc in local}(ii), the $q$-binomial identity \eqref{q-binomial} and \lemref{loc in loc}(iii).  \endproof

\subsection{Filtration of $W_{loc}(m\omega_i)\otimes W_{loc}(k\omega_i)$ for $i=1,n$}

\begin{theorem} For $m,k \in \mathbb{Z}_{+}$, let $\mathcal M, L\in \mathbb N$ be as described in \lemref{tensor product}. Then,
\begin{itemize}
 \item[i.]  $W_{loc}(m\omega_{1}) \otimes W_{loc}(k\omega_{1})$ has a filtration by $\mathfrak{sl}_{n+1}[t]$-modules where the successive quotients are isomorphic to $$\tau_{l_r}M((\mathcal M-L)\omega_1 + r \omega_2, 2(L-r)\omega_1 ), \qquad 0\leq r\leq L, \quad 0\leq l_r\leq (L-r)r.$$
 \item[ii.] $W_{loc}(m\omega_{n}) \otimes W_{loc}(k\omega_{n})$ has a filtration by  $\mathfrak{sl}_{n+1}[t]$-modules where successive quotients are isomorphic to $$\tau_{l_r}M(r\omega_{n-1} + (\mathcal M-L) \omega_n , 2(L-r)\omega_n ) \qquad 0\leq r\leq L, \quad 0\leq l_r\leq (L-r)r.$$
\end{itemize}
\end{theorem}
\begin{proof} The proof of this theorem is similar to proof of 
\thmref{filter.1.2}.\\
(i). Given $r \in \mathbb{N}$, let $w_r$ be the generators of $W_{loc}(r\omega_{1})$. Then by \propref{POP}, $\mathfrak{B}(r\omega_1)=\{y(\eta,\pi_\eta)w_r : (\eta,\pi_{\eta})\in \mathfrak P(r\omega_1)\}$ is an ordered basis of $W_{loc}(r\omega_1)$ where $\eta$ ranges over Gelfand Tsetlin patterns of the form $$\eta(\bop) := \begin{array}{ccccccccc} &&&&p_1&&&&\\
 &&&p_2&&0&&&\\ &&\cdots&&&&\cdots&&\\ &p_n&0&&&&&\cdots&\\ k&0&&&\cdots&&&\cdots&0
\end{array}, \text{ where } \bop = (p_n\geq p_{n-1}\geq \cdots\geq p_1).$$

\noindent Assuming $m\geq k$, for each $(\eta,\pi_\eta)\in \mathfrak{P}(k\omega_1)$, set 
\begin{align*}
v_{(\eta,\pi_\eta)}= w_m\otimes y(\eta,\pi_\eta)w_k, \qquad \qquad \qquad \\
\mathfrak B_{m,k}^{(\eta,\pi_\eta)\leq} = \{v_{(\nu,\pi_\nu)} : (\nu,\pi_\nu)\in \mathfrak{P}(k\omega_1),\, (\eta,\pi_\eta)\leq (\nu,\pi_\nu)\},
\end{align*} and let $\mathcal S(\eta,\pi_\eta)$ be the submodule of $W_{loc}(m\omega_1)\otimes W_{loc}(k\omega_1)$ generated by $\mathfrak B_{m,k}^{(\eta,\pi_\eta)\leq}$. By \lemref{generator of tensor}, $$\mathcal S(\eta_0,\pi_{\eta_0})= W_{loc}(m\omega_1)\otimes W_{loc}(k\omega_1),$$ where $\eta_0 =\eta(\bop)$ with $\bop= (k,k,\cdots,k,0)\in \mathbb N^n$. Furthermore, from the definition it follows that, 
$\mathcal S(\nu,\pi_\nu)$ is a submodule of  $\mathcal S(\eta,\pi_\eta)$ 
whenever $(\eta,\pi_\eta) \leq (\nu,\pi_\nu)$. This gives a descending chain of submodules of the $\mathfrak{sl}_{n+1}[t]$-module $W_{loc}(m\omega_1)\otimes W_{loc}(k\omega_1)$. Consequently, $$\ch_{gr} W_{loc}(m\omega_1)\otimes W_{loc}(k\omega_1) = \sum\limits_{(\eta,\pi_\eta)\in \mathfrak P(k\omega_1)} \ch_{gr} \dfrac{\mathcal S(\eta,\pi_\eta)}{ \mathcal S(\eta^+,\pi_\eta^+)}, $$ where $(\eta^+,\pi_{\eta^+})\in \mathfrak P(k\omega_1)$ is, as defined in \secref{POP}. 

Now, to prove the theorem we consider the modules $\dfrac{\mathcal S(\eta,\pi_\eta)}{ \mathcal S(\eta^+,\pi_\eta^+)}$ for $(\eta^+,\pi_{\eta^+})\in \mathfrak P(k\omega_1)$. First of all notice, as $$y_{2i}(r,\bor)y_{11}^{(r)}w_{k\omega_1} = y_{1i}(r,\bor)w_{k\omega_1}, \qquad \forall\, i\in I, \quad \bor\in \mathbb N^r,$$  whenever $\nu= \eta(\bop)$ with $\bop=(p_n\geq p_{n-1}\geq \cdots\geq p_1)\in \mathbb N^n$ and $p_i\neq k$ for some $2\leq i\leq n$,
$$v_{\eta,\pi_{\eta_n}} \in \mathcal S(\nu,\pi_\nu), \quad \text{ for } \nu = \eta(\bos)\, \text{ with } \bos =(k,k,\cdots,k,k-p_1)\in \mathbb N^n.$$ Thus, ${\mathcal S(\eta,\pi_\eta)}\subseteq { \mathcal S(\eta^+,\pi_\eta^+)}$ for all $\eta=\eta(\bop)$, whenever $\bop =(p_n,p_{n-1},\cdots,p_1)\in \mathbb N^n$ with $p_i\neq k$ for some $2\leq i\leq n$. Consequently,
\begin{equation}\label{ch.1.1}\ch_{gr} W_{loc}(m\omega_1)\otimes W_{loc}(k\omega_1) = \sum\limits_{(\eta_p,\pi_{\eta_p})\in \mathfrak P(k\omega_1)} \ch_{gr} \dfrac{\mathcal S(\eta_p,\pi_{\eta_p})}{ \mathcal S(\eta_p^+,\pi_{\eta_p}^+)}, \end{equation} where 
$\eta_p = \eta(\bop)$ with $\bop=(k,k,\cdots,k,p)\in \mathbb N^n$ for $0\leq p\leq k$.
 
Similar arguments as those used in \thmref{filter.1.2}, now show that 
$\dfrac{\mathcal S(\eta,\pi_\eta)}{ \mathcal S(\eta^+,\pi_\eta^+)}$
is a quotient of $W_{loc}((m+k)\omega_1-p\alpha_1)$, whenever $\eta = \eta_p$
with $0\leq p\leq k$.
By definition of $\mathfrak P(k\omega_1)$, for a fixed $r\in \mathbb N$, there 
exists a one-to-one correspondence between the sets 
 $\{(\eta_p, \pi_{\eta_p})\in \mathfrak P(k\omega_1) : |(\eta_p, \pi_{\eta_p})|_{11}=r \}$ and $\{(s_1\leq \cdots\leq s_p)\in \mathbb Z^k: 0\leq s_i\leq k-p,\, \sum\limits_{i=1}^p s_i=r \}$. So from standard  q-binomial theory it follows that 
$$\# \{(\eta_p, \pi_{\eta_p})\in \mathfrak P(k\omega_1) : |(\eta_p, \pi_{\eta_p})|_{11}=r \} = \text{coefficient of } q^r \text{ in } \begin{bmatrix} k \\ p \end{bmatrix}_{q}. $$
Now, comparing the \eqref{ch.1.1} with the graded character \eqref{ch.M.n.l} of $W_{loc}(m\omega_1)\otimes W_{loc}(k\omega_1)$  and observing that $\dfrac{\mathcal S(\eta_p,\pi_{\eta_p})}{ \mathcal S(\eta_p^+,\pi_{\eta_p}^+)}$ is a quotient of $W_m((m+k)\omega_1-p\alpha_1)$, we conclude that the $\mathfrak{sl}_{n+1}[t]$-module $\dfrac{\mathcal S(\eta_p,\pi_{\eta_p})}{ \mathcal S(\eta_p^+,\pi_{\eta_p}^+)}$ is isomorphic to  $ M((m-k)\omega_1+p\omega_1,2(k-p)\omega_1)$ for $1\leq p\leq k$. Since  by the characterization of $\mathfrak P(k\omega_1)$,  $|(\eta(p), \pi_{\eta(p)})|_{11}$ ranges between $0$ and $(k-r)r$ and the theorem follows in the case when $m\geq k$.

\noindent (ii) The proof in this case can be established in the same way using the basis $\mathfrak B(k\omega_n)$ of $W_{loc}(k\omega_n)$ and the graded character \eqref{ch.M.n.l.2} of $W_{loc}(m\omega_n)\otimes W_{loc}(k\omega_n)$.
\end{proof}

\begin{remark} It has been proved in \cite{23} that for $n\leq 3$ the 
 the $\mathfrak{sl}_{n+1}[t]$-modules $M(\nu,\lambda)$ have a filtration by level 2 Demazure modules. Coupled with our results, this shows that the tensor product $\mathfrak{sl}_{n+1}[t]$-modules $W_{loc}(m\omega_i)\otimes W_{loc}(k\omega_i)$ for $i=1,2$ have a filtration by level 2 Demazure modules for $n\leq 3$.
\end{remark}

%
%
%
%
%
%
\section{The Modules $M_j(\mu,\nu,\rho)$ and proof of \thmref{trunc}($i$)} \label{M_j(m.n.theta)}
In this section we complete the proof of \thmref{trunc}(i). For this purpose we consider a class of CV modules  for $\mathfrak{sl}_3[t]$,  $M_j(\mu,\nu,\rho)$, which are indexed by a non-negative integer $j$ and a triple of dominant integral weights $(\mu,\nu,\rho)$. 

\subsection{} Given a non-negative integer $j$ and  $\boldsymbol{\lambda}  = (\lambda_1,\lambda_2,\lambda_3)\in (P^+)^3$,  define a $\Phi^+$-tuple of partitions $\xi_{\Blambda,j}=(\xi^1,\xi^\theta,\xi^2)$ as follows :
$$\begin{array}{c}
\xi^\theta = (3^{\lambda_3(h_\theta)},2^{\lambda_2(h_\theta)+j},1^{\lambda_1(h_\theta)}), \qquad 
\quad \xi^i=(3^{\lambda_3(h_i)},2^{\lambda_2(h_i)},1^{\lambda_1(h_i)+j}), \qquad i\in I.
\end{array}$$
Let 
$M_j(\lambda_1, \lambda_2,\lambda_3)$ be the CV module $V(\xi_{\boldsymbol{\lambda},j})$ for $\mathfrak{sl}_3[t]$-module generated by a non-zero vector $v_{\boldsymbol{\lambda},j}$. Let $\uplambda=\lambda_1+2\lambda_2+3\lambda_3$. By definition of CV modules $v_{\Blambda,j}$ satisfies the following relations :
$$\mathfrak n^{+}[t] v_{\Blambda,j} =0; \quad (h \otimes t^{r})v_{\Blambda,j} = \delta_{r,0} (\uplambda+j\theta)(h)v_{\Blambda,j} \quad (y_{i} \otimes 1)^{(\uplambda +j\theta)(h_i)+1}v_{\lambda,j} =0, \quad  i \in I, r \in \mathbb{Z}_{+} , $$ 
\begin{align}
(y_{i} \otimes t^{(\lambda_1 + \lambda_2 +\lambda_3 +j\theta)(h_i)})v_{\Blambda,j} & =0,\quad  i \in I,\label{M.rel.1}\\
(y_\theta \otimes t^{(\lambda_1 + \lambda_2 + \lambda_3 +\frac{j\theta}{2})(h_\theta)}) v_{\Blambda,j} &= 0,\label{M.rel.2}\\
y_{\alpha}(2, (\lambda_1+2\lambda_2 + 2\lambda_3 +j\theta)(h_{\alpha})-1)v_{\Blambda,j} &= 0, \quad \alpha \in R^{+} \label{M.rel.3}\\
y_{\alpha}(3, (\lambda_1+2\lambda_2 + 3\lambda_3 +j\theta)(h_{\alpha})-2)v_{\Blambda,j} &= 0, \quad \alpha \in R^{+}.\label{M.rel.4}
\end{align} 

\label{Mjmod}
\noindent Comparing the definitions of $M_j(\lambda_1,\lambda_2,\lambda_3)$, $M(\mu,\nu)$-modules and $D(\ell, \ell\mu)$, we see that:
\begin{align*} 
M_0(\lambda_1,0,0)\cong_{\mathfrak{sl}_3[t]} W_{loc}(\lambda_1), \qquad & \qquad
M_1(0,\lambda_2,0)\cong_{\mathfrak{sl}_3[t]} D(2,2\lambda_2+\theta),\\
M_1(0,0,\lambda_3)\cong_{\mathfrak{sl}_3[t]} D(3,3\lambda_3+\theta), \qquad & \qquad 
M_0(\lambda_1,\lambda_2,0)\cong_{\mathfrak{sl}_3[t]} M(\lambda_1,2\lambda_2)
\end{align*}
\begin{remark} \label{rem.trunc} Observe that when $\lambda_2=\lambda_3=0$, the relations \eqref{M.rel.1} and \eqref{M.rel.3} are a consequence of \eqref{M.rel.2}. As a result, if $\lambda_1-j\theta\in P^+$, the module $M_j(\lambda_1-j\theta,0,0)$   is a quotient of $W_{loc}(\lambda_1)$ by the submodule generated by $(y_\theta \otimes t^{(\lambda_1-j\theta +\frac{j\theta}{2})(h_\theta)}) w_{\lambda_1}$. Hence, $M_j(\lambda_1-j\theta,0,0)$ is isomorphic 
$W_{|\lambda_1|-j}(\lambda_1)$ for each $0\leq j\leq \min\{\lambda_1(h_i): i=1,2\}$. To complete the proof of \thmref{trunc}(i), it therefore suffices to prove that $M_j(\lambda_1,0,0)$ is isomorphic as a $\mathfrak{sl}_3[t]$-module to $D(1,\lambda_1)\ast V(\theta)^{\ast j}$.   \end{remark}

We deduce this from the following theorem which is the main result of this section. 

\begin{theorem} \label{Main.thm.8} 
Given a triple of dominant integral weights $(\lambda_1,\lambda_2,\lambda_3)$,
\begin{enumerate}
\item[i.] $M_0(\lambda_1,\lambda_2,\lambda_3)  \cong_{\mathfrak{sl}_3[t]} D(3,3\lambda_3) \ast D(2,2\lambda_2)*D(1,\lambda_1)$.
\item[ii.] $M_j(\lambda_1, \lambda_2,\lambda_3) \cong_{\mathfrak{sl}_3[t]} D(3,3\lambda_3) \ast D(2,2\lambda_2)*D(1,\lambda_1)*V(\theta)^{*j}$,
for $j\geq 1$, when $\lambda_3=0$ or $|\lambda_1(h_\theta)|\leq 1$.
\end{enumerate}
\end{theorem} 

\noindent Beginning with \lemref{dim V geq}, we now proceed to prepare for the proof of \thmref{Main.thm.8}.

\begin{lemma}\label{dim V geq}    There exist surjective $\mathfrak{sl}_3[t]$-module homomorphism.
 $$\delta: M_j(\lambda_1,\lambda_2,\lambda_3) \longrightarrow D(3,3\lambda_3)^{a_1}\ast D(2,2\lambda_2)^{a_2} \ast D(1, \lambda_1)^{a_3} \ast V(\theta)^{a_4} \ast  \cdots \ast V(\theta)^{a_{j+3}}$$ for all $\boa=(a_1,a_2,\cdots,a_{j+2},a_{j+3})\in \mathbb C^{j+3}$,  such that $a_i\neq a_j$ for $i\neq j$.
Consequently, $\dim M_j(\lambda_1,\lambda_2,\lambda_3) \geq (\prod\limits_{j=1}^3\dim D(j,j\lambda_j)) (\dim V(\theta))^{j}.$
\end{lemma}
\begin{proof} Let $v_j \in D(j,j\lambda_j)$ be the image of highest weight vector $w_{j\lambda_j}$ of $W_{loc}(j \lambda_j)$ for $j=1,2,3$ and let $v_\theta $ be the highest weight vector of $V(\theta).$
From the defining relations of $D(3, 3\lambda_3)$, $D(2, 2\lambda_2)$, $D(1,\lambda_1)$, and $V(\theta)$ it follows that,
 \begin{align}
     &(y_i \otimes t^{\lambda_j(h_i)}) v_j=0,\,   \, (y_{\theta} \otimes t^{\lambda_j(h_{\theta})})v_j=0, \qquad \text{ for } 1\leq j\leq 3, \, i\in I \label{i1}  \\ 
      &(y_{\alpha} \otimes t) v_\theta=0, \quad \forall \alpha \in \Phi^+, \, (y_{i} \otimes 1)^2v_\theta =0 \quad \forall i\in I \label{i2}
  \end{align}
Hence there exists a surjective $\mathfrak{sl}_3[t]$-module homomorphism,
  $$W_{loc}(3\lambda_3+2\lambda_2+ \lambda_1+ j\theta) \longrightarrow D(3,3\lambda_3)^{a_1}\ast D(2,2\lambda_2)^{a_2} \ast D(1,\lambda_1)^{a_3} \ast V(\theta)^{a_4} \ast \cdots \ast V(\theta)^{a_{j+3}}.$$ Further, using \lemref{fusion_def} and the relations \eqref{i1} and \eqref{i2}, we see that :
\begin{equation} \begin{array}{l}
      (y_{i} \otimes t^{(\lambda_1 + \lambda_2 +\lambda_3 +j\theta)(h_i)})(v_3 \ast v_2 \ast v_1 \ast \underbrace{v\ast \cdots \ast v}_{j-times})=0, \quad \text{ for } i\in I\\
     (y_{\theta} \otimes t^{(\lambda_1 + \lambda_2 + \lambda_3 +\frac{j\theta}{2})(h_{\theta})})(v_3 \ast v_2 \ast v_1 \ast \underbrace{v\ast \cdots \ast v}_{j-times})=0 \label{i2p3}\end{array}\end{equation} 
 Thus to complete the proof of the lemma,  it suffices to show that:
\begin{eqnarray}
     &&y_{\alpha}(2, (\lambda_1 + 2\lambda_2 +2\lambda_3 +j\theta)(h_{\alpha})-1)(v_3 \ast v_2 \ast v_1 \ast \underbrace{v\ast \cdots \ast v}_{j-times})=0, \ \ \text{ for } \, \alpha \in \Phi^{+}. \qquad \qquad \label{i2p4}  \end{eqnarray}
Owing to \eqref{i2p3}, we see that for $i\in I$, the left hand side of \eqref{i2p4} is equal to \begin{equation}\sum\limits_{k=0}^{\lambda_1(h_i)+j-1}(y_{i} \otimes t^{(\lambda_2+\lambda_3)(h_i)+k}) (y_{i} \otimes t^{(\lambda_1+\lambda_2+\lambda_3)(h_i)+j-1-k})(v_3 \ast v_2 \ast v_1 \ast \underbrace{v\ast \cdots \ast v}_{j-times})\label{fusion.M.1.2.3}\end{equation} Further, since for $s=2,3$,  $(y_i \otimes t^r) v_s=0$, for $r\geq \lambda_s(h_i)$ and $y_\alpha\otimes t^r v=0$, for $r>0$, 
 we see that using \lemref{fusion_def},
\eqref{fusion.M.1.2.3} is equal to :
$$v_3 \otimes v_2\otimes (Y.v_1) \otimes \underbrace{v \otimes \cdots \otimes v}_{j-times}
$$ where
\begin{align*}
     Y.v_1 =
\sum\limits_{k=0}^{\lambda_1(h_i)+j-1} (y_{1} \otimes (t-a_1)^{k}
\prod\limits_{p=2}^{3}(t-a_p)^{\lambda_p(h_i)})(y_{1} \otimes (t-
a_1)^{\lambda_1(h_i)+j-1-k}\prod\limits_{p=2}^3
 (t-a_p)^{\lambda_p(h_i)} )(v_1) 
\end{align*}
However, from the defining relations of $D(1,\lambda_1)$ we know that 
$$y_i(2, \lambda_1(h_i)+2p)v_1 =0, \qquad \forall \, p\geq 0.$$ Therefore using the relation $y_i(1, \lambda_1(h_i)+2p)v_1 =0$ for $p\geq 0$, we have 
$\sum\limits_{k=0}^{\lambda_1(h_i) -1}(y_i \otimes t^{k+p})(y_i \otimes t^{\lambda_1(h_i)-k+p})= 0, \text{ for } p \geq 0$. In particular, this gives, 
      $$ \sum\limits_{k=0}^{\lambda_1(h_i)+j-1} (y_{1} \otimes (t)^{k}
\prod\limits_{p=2}^{3}(t+a_1-a_p)^{\lambda_p(h_i)})(y_{1} \otimes (t)^{\lambda_1(h_i)+j-1-k}\prod\limits_{p=2}^3  (t+a_1-a_p)^{\lambda_p(h_i)} )(v_1)
  \in     F^{K_i-1}(V),$$ where $K_i=\lambda_1 + 2\lambda_2 +2\lambda_3 +j\theta)(h_i)-1$ and 
$$ F^{K_i-1}(V) = \sum\limits_{r=0}^{K_i-1}\bu(\mathfrak{sl}_3[t])[r]
v_3\otimes v_2\otimes v_1\otimes 
\underbrace{v \otimes \cdots \otimes v}_{j-times}.$$
 Thus using the definition of fusion product we see that \eqref{i2p4} follows 
for $\alpha=\alpha_i, i\in I$. Using similar arguments, it can be shown that  
\eqref{i2p4} holds for all $\alpha \in \Phi^+$. This completes the proof of the 
lemma. \end{proof}

\subsection{The module $M_0(\lambda_1,\lambda_2,\lambda_3)$} 

\begin{lemma}\label{naoi.cv} Let $\mu, \nu, \rho$ be dominant integral weights of $\mathfrak{sl}_3(\mathbb C).$ Then the fusion product module,  $D(3,3\rho)*D(2,2\nu)*D(1,1\mu)$ is cyclic $\mathfrak{sl}_3[t]$-module generated by a vector $v_{\xi}$ with the following defining relations :
\begin{equation}\label{naoi.realization}
\begin{array}{rl}
\mathfrak n^{+}[t] v_{\xi} =0; \quad (h \otimes t^{r})v_{\xi} &= \delta_{r,0} (\mu+2\nu+3\rho)(h)v_{\xi}, \qquad  r \in \mathbb{Z}_{+} \\ (y_{i} \otimes 1)^{(\mu+2\nu+3\rho)(h_i)+1}v_{\xi} &=0, \qquad  i \in I, \\
(y_{i} \otimes t^{(\mu+\nu+\rho)(h_i)})v_{\xi} & =0,\qquad  i \in I,\\
y_{i}(2, (2\rho + 2\nu +\mu)(h_{i})-1)v_{\xi} &= 0, \qquad i \in I.
\end{array}\end{equation} \end{lemma}  
\proof In the special case when  $\mu, \nu, \rho$ are fundamental weights of $\mathfrak{sl}_3(\mathbb C)$, the lemma is a restatement of \cite[Theorem B]{MR3704743}. On the other hand, for an arbitrary dominant integral weight $\lambda= \sum\limits_{i=1}^2 \ell m_i\omega_i$ of $\mathfrak{sl}_3(\mathbb C)$, it follows from \cite[Proposition 5.1]{MR3296163}  that  
$$D(\ell,\lambda) \cong_{\mathfrak{sl}_3[t]} D(\ell, \ell\omega_1)^{\ast m_1} \ast D(\ell, \ell\omega_2)^{\ast m_2}.$$ The two results \cite[Theorem B]{MR3704743} and \cite[Proposition 5.1]{MR3296163} together imply the lemma.\endproof

\noindent {\bf{Proof of}} \thmrefb{Main.thm.8}(i): Using the presentation of  $D(3,3\lambda_3)*D(2,2\lambda_2)*D(1,\lambda_1)$ given by \eqref{naoi.realization} and the  
 defining relations of $M_{0}(\lambda_1,\lambda_2,\lambda_3)$, it is easy to see that there exists a  surjective $\mathfrak{sl}_3[t]$-module homomorphism 
$$\alpha^{+}: D(3,3\lambda_3)*D(2,2\lambda_2)*D(1,\lambda_1)\longrightarrow M_{0}(\lambda_1,\lambda_2,\lambda_3)$$ such that $\alpha^+(v_{\xi})=v_{\lambda,0}$.
On the other hand by \lemref{dim V geq}, the fusion product module, $D(3,3\lambda_3)*D(2,2\lambda_2)*D(1,\lambda_1)$ is a quotient of $M_0(\lambda_1, \lambda_2,\lambda_3)$. Hence it follows that, 
\begin{equation} \dim M_0(\lambda_1, \lambda_2,\lambda_3) = \dim D(3,3\lambda_3) \dim D(2,2\lambda_2) \dim D(1,\lambda_1) \label{M_0.dim}\end{equation} and consequently, 
$$M_0(\lambda_1, \lambda_2,\lambda_3)\cong_{\mathfrak{sl}_{n+1}[t]} D(3,3\lambda_3)*D(2,2\lambda_2)*D(1,\lambda_1).$$ This completes the proof of \thmref{Main.thm.8}(i).  \endproof

\subsection{The modules $M_1(\lambda_1,\lambda_2,\lambda_3)$ for $|\lambda_1|\leq 1$}
%
%
%
%
\begin{theorem}\label{dim_V_geq.1} Let $\lambda_1 ,\lambda_2, \lambda_3 \in P^+$
with $\lambda_1(h_\theta)\leq 1$. \\
(i). When $\lambda_1=0$ and  $\lambda_2(h_1) \geq 1$, we have the short exact sequence :  \begin{equation} \label{sh.1.0} 0 \longrightarrow \tau_{s_1}M_0(\omega_1,\lambda_2 -\omega_1 +\omega_2,\lambda_3) \longrightarrow M_{1}(0,\lambda_2,\lambda_3) \overset{\phi^+}{\longrightarrow} M_0 (\omega_2, \lambda_2 -\omega_1,\lambda_3 + \omega_1)\rightarrow 0\end{equation} where $s_1= (\lambda_2+\lambda_3)(h_1)$. \\
\noindent (ii). When $\lambda_1=\omega_1$, we have the short exact sequence : \begin{equation} \label{sh.2.0} 0 \longrightarrow \tau_{s_1} M_0(0, \lambda_2 +\omega_2, \lambda_3) \longrightarrow M_1 (\omega_1, \lambda_2,\lambda_3) \overset{\psi^+}{\longrightarrow} M_0 (\omega_2, \lambda_2 +\omega_1, \lambda_3)\rightarrow 0, \end{equation} 
where $s_1= (\lambda_2+\lambda_3)(h_1)+1$.
\end{theorem}
\begin{proof}
    (i) From the defining relations of  $M_{1}(0,\lambda_2,\lambda_3)$ and $M_0 (\omega_2, \lambda_2 -\omega_1,\lambda_3 + \omega_1)$, it is clear that $\phi^+$ is a surjective $\mathfrak{sl}_3[t]$-module homomorphism 
with  $$\ker \phi^+= \left\langle (y_1 \otimes t^{(\lambda_2+\lambda_3)(h_1)})v_{\Blambda,1}, (y_{\theta} \otimes t^{(\lambda_2+\lambda_3)(h_{\theta})})^2 v_{\Blambda,1} \right\rangle .$$
By the defining relations of $M_1(0,\lambda_2,\lambda_3)$ we have, 
\begin{equation}\label{M.j.0.2.3}
y_i\otimes t^{(\lambda_2+\lambda_3)(h_i)+1}.v_{\Blambda,1}=0, \qquad y_i(2,2(\lambda_2+\lambda_3))(h_i))v_{\Blambda,1} =0 , \qquad \text{ for } i\in I. \end{equation}
 Therefore, $\dfrac{(y_{1} \otimes t^{(\lambda_2+\lambda_3)(h_{1})})^2}{2}\dfrac{(y_{2} \otimes t^{(\lambda_2+\lambda_3)(h_{2})})^2}{2} v_{\Blambda,1}=0$ and consequently, 
$$\dfrac{(y_{\theta} \otimes t^{(\lambda_2+\lambda_3)(h_{\theta})})^2}{2}.v_{\Blambda,1} = 
(y_{2} \otimes t^{(\lambda_2+\lambda_3)(h_{2})})(y_{\theta} \otimes t^{(\lambda_2+\lambda_3)(h_{\theta})})(y_{1} \otimes t^{(\lambda_2+\lambda_3)(h_{1})})v_{\Blambda,1},$$ implying,
$\ker \phi^+ = \bu(\mathfrak{sl}_3[t])(y_{1} \otimes t^{(\lambda_2+\lambda_3)(h_{1})})v_{\Blambda,1}.$ It is clear that 
$(y_{1} \otimes t^{(\lambda_2+\lambda_3)(h_{1})})v_{\Blambda,1}$ 
is a highest weight vector of weight $3\lambda_3 +2\lambda_2 -\omega_1 + 2\omega_2$ which is annihilated by $h\otimes t^r$ for all $h\in \mathfrak h$ and $r\in \mathbb N$. Further, from the relations \eqref{M.j.0.2.3} and \eqref{M.rel.2}
it follows that for $\lambda_1=0$, $j=1$,       
\begin{align*}
(y_{1} \otimes t^{(\lambda_2+\lambda_3)(h_{1})})(y_{1} \otimes t^{(\lambda_2+\lambda_3)(h_{1})})v_{\Blambda,1}& = 0,\\
y_{1}(2,2((\lambda_2+\lambda_3)(h_{1})-1))(y_{1} \otimes t^{(\lambda_2+\lambda_3)(h_{1})})v_{\Blambda,1}& =0,\\
    (y_{2}\otimes t^{(\lambda_2+\lambda_3)(h_{2})+1})(y_{1} \otimes t^{(\lambda_2+\lambda_3)(h_{1})})v_{\Blambda,1} & = 0. 
    \end{align*}
Therefore, by \lemref{naoi.cv} we have a surjective map from $\tau_{s_1}M_0(\omega_1, \lambda_2-\omega_1+\omega_2, \lambda_1)$ to $\ker \phi^+$. Consequently
\begin{align*}
    \dim M_1 (0, \lambda_2,\lambda_3) & = \dim Ker(\phi^+) + \dim M_0(\omega_2, \lambda_2 -\omega_1,\lambda_3+\omega_1) \\
    &\leq \dim M_0(\omega_1,\lambda_2 -\omega_1 +\omega_2,\lambda_3) + \dim M_0(\omega_2, \lambda_2 - \omega_1, \lambda_3+\omega_1)\\
    & = \dim (D(3,3\lambda_3)*D(2,2(\lambda_2 - \omega_1 + \omega_2) )* D(1,\omega_1)) \\
    & + \dim (D(3,3(\lambda_3+\omega_1) * D(2,2(\lambda_2 - \omega_1)) * D(1,\omega_2)) \quad (\text{By \thmref{Main.thm.8}(i)})\\
    &\leq 10^{\lambda_3(h_{\theta})}6^{\lambda_2(h_{\theta})}3 + 10^{\lambda_3(h_{\theta})+1}6^{\lambda_2(h_{\theta})-1}3 = 10^{\lambda_1(h_{\theta})}6^{\lambda_2(h_{\theta})}8\\
&=\dim (D(3,3\lambda_3)*D(2,2\lambda_2)*V(\theta)
\end{align*}
On the other hand, by \lemref{dim V geq}, $\dim M_1(0,\lambda_2,\lambda_3) \geq \dim (D(3,3\lambda_3)*D(2,2\lambda_2)*V(\theta))$. From the two inequalities it follows that $\dim M_1(0,\lambda_2,\lambda_3) \geq \dim (D(3,3\lambda_3)*D(2,2\lambda_2)*V(\theta))$ which implies $\dim \ker \phi^+ = \dim M_0(\omega_1, \lambda_2-\omega_1+\omega_2, \lambda_1)$. This proves part(i) of the theorem.

\noindent (ii). From the defining relations of $M_1 (\omega_1, \lambda_2,\lambda_3)$ and $M_0 (\omega_2, \lambda_2 +\omega_1,\lambda_3)$, it is clear that $\psi^+$ is surjective with $$\ker \psi^+= < y_1 \otimes t^{(\lambda_3+\lambda_2)(h_1)+1} v_{\Blambda,1} >.$$ 
By the defining relations of $M_1(\omega_1,\lambda_2,\lambda_3)$ we have, 
\begin{equation*}\label{M.j.1.2.3}
y_i\otimes t^{(\lambda_2+\lambda_3+\omega_1)(h_i)+1}.v_{\Blambda,1}=0, \qquad y_i(2,2(\lambda_2+\lambda_3))(h_i)+\omega_1(h_i))v_{\Blambda,1} =0 , \qquad \text{ for } i\in I. \end{equation*}
It easily follows from the above relations that $y_1 \otimes t^{(\lambda_3+\lambda_2)(h_1)+1} v_{\Blambda,1}$ is a highest weight vector of weight $ 3\lambda_3+2(\lambda_2+\omega_2)$ which is annihilated by $h\otimes t^r$ for all $h\in \mathfrak h$ and $r\in \mathbb N$. Now similar arguments as used in part(i) of the theorem, show that, 
\begin{equation}\label{dim V(mu)}
    \dim M_1(\omega_1,\lambda_2,\lambda_3) \leq \dim (D(3,3\lambda_3)*D(2,2\lambda_2)*D(1,\omega_1)*V(\theta).
\end{equation} As is part(i), using \lemref{dim V geq} and \eqref{dim V(mu)}, we see that the modules $\ker \psi^+$ and $M_0(0,\lambda_2+\omega_2,\lambda_3)$ are isomorphic and this completes the proof of the theorem. \end{proof}

\vspace{.15cm}

\noindent Using the order 2, Dynkin diagram automorphism $w_0$ of $\mathfrak{sl}_3$ that maps $\alpha_1$ to $\alpha_2$, the following can be deduced from \thmref{dim_V_geq.1}.

\begin{corollary}\label{dim_V_geq.1} Let $\lambda_1 \lambda_2, \lambda_3 \in P^+$ with $\lambda_1(h_\theta)\leq 1$. \\
(i). If $\lambda_2(h_1) =0$ and $\lambda_2(h_2)\geq 1$, then we have the short exact sequence : $$0 \rightarrow \tau_{s_1^{'}}M_0 (\omega_2,\lambda_2 -\omega_2 + \omega_1,\lambda_3) \rightarrow M_1 (0,\lambda_2,\lambda_3) \rightarrow M_0 (\omega_1,\lambda_2 -\omega_2,\lambda_3 +\omega_2) \rightarrow 0$$ where $s_1^{'} = \lambda_3(h_2)+\lambda_2(h_2)$.\\
(ii). If $\lambda_1 =\omega_2$, then we have the short exact sequence :
 $$0 \rightarrow \tau_{s_3}M_0 (0, \lambda_2 +\omega_2,\lambda_3) \rightarrow 
M_1 (\omega_1, \lambda_2,\lambda_3) \rightarrow M_0 (\omega_2,\lambda_2 +\omega_1,\lambda_3) \rightarrow 0,\, $$
    where $s_{2} = \lambda_2 (h_2)+\lambda_3 (h_2)+1$ and $s_{3} = \lambda_2 (h_1)+\lambda_3 (h_1)+1$.
\end{corollary}

\subsection{The modules $M_j(\lambda_1,\lambda_2,\lambda_3)$ when $|\lambda_1|\leq 1$ and $j\geq 2$} 

\begin{theorem}\label{jgeq2} Given $\lambda_1 ,\lambda_2, \lambda_3 \in P^+$ with $\lambda_1(h_\theta)\leq 1$, For $j\geq 2$, $M_j(\lambda_1,\lambda_2,\lambda_3)$ admits the short exact sequence :
 \begin{equation} 0 \longrightarrow \ker \phi_j^+ \longrightarrow M_j(\lambda_1, \lambda_2, \lambda_3) \overset{\phi_j^+}{\longrightarrow} M_{j-2}(\lambda_1,\lambda_2 +\theta, \lambda_3) \longrightarrow 0\end{equation} 
where $\ker\phi_j^+$ has a filtration $0 \subset U_{1}^j \subset U_{2}^j = \ker \phi_j^+$ such that, \\
(i). For $\lambda_1=0$,
$${U_1^j } \cong \tau_{s_{1,j}}M_{j-2}(\omega_2, \lambda_2+\omega_2,\lambda_3), \quad \frac{U_2^j}{U_1^j} \cong \tau_{s_{2,j}}M_{j-2}(0, \lambda_2, \lambda_3+\omega_1), $$ with $s_{1,j} = (\lambda_3+\lambda_2)(h_1)+j-1,  s_{2,j} = (\lambda_3+\lambda_2)(h_2) +j-1$.\\ 
(ii). For $\lambda_1=\omega_1$, 
$$U_1^j \cong \tau_{s_{3,j}}M_{j-1}(0, \lambda_2+\omega_2, \lambda_3), \quad \frac{U_2^j}{U_1^j} \cong \tau_{s_{4,j}}M_{j-2}(0, \lambda_2 +2\omega_1, \lambda_3) $$ with $s_{3,j} = (\lambda_2+\lambda_3)(h_1) +j$ and $ s_{4,j} = (\lambda_2+\lambda_3)(h_2) +j-1$.\\
Consequently, for $j\geq 2$, $ M_{j}(\lambda_1,\lambda_2,\lambda_3) \cong_{\mathfrak{g}[t]} D(3,3\lambda_3)*D(2,2\lambda_2)*D(1,\lambda_1)*V(\theta)^{*j}$ whenever $\lambda_1 =0$ or $\lambda_1=\omega_1$.\label{Case II}
\end{theorem}
\begin{proof} We prove the theorem by induction on $j$.

\noindent For $\lambda_1=0$,  from the defining relations of $M_j(0,\lambda_2,\lambda_3)$ and $M_{j-2}(0,\lambda_2+\theta,\lambda_3)$, it is easy to see that there exists a $\mathfrak{sl}_3[t]$-module epimorphism 
$M_j(0,\lambda_2,\lambda_3) \xrightarrow{\phi_j^+} M_{j-2}(0,\lambda_2+\theta,\lambda_3)$ such that 
    $$\ker \phi_j^+ =\left\langle (y_1 \otimes t^{(\Blambda_2+\lambda_3)(h_1) + j -1})v_{\lambda,j}, (y_2 \otimes t^{(\lambda_2+\lambda_3)(h_2) + j-1})v_{\Blambda,j}\right\rangle. $$
Setting, $$ U_1^j = 
U(\mathfrak{g}[t])(y_1 \otimes t^{(\lambda_2+\lambda_3)(h_1) + j -1})v_{\Blambda,j},
\qquad
U_2^j = U(\mathfrak{g}[t])(y_2 \otimes t^{(\lambda_2+\lambda_3)(h_2) + j-1})v_{\Blambda,j} 
+ U_1^j,$$ we see that $0 \subset U_1^j \subset U_2^j = \ker \phi^+_j$ 
is a filtration of $\ker \phi_j^+$.\\
\noindent By the defining relations of $M_j(0,\lambda_2,\lambda_3)$ for all integers $p\geq 0$ we have,
\begin{equation} \begin{array}{ll} 
y_i(1, (\lambda_2+\lambda_3)(h_i)+j+p)v_{\Blambda,j}=0,  & 
y_\theta(1, (\lambda_2+\lambda_3)(h_\theta)+j+p)v_{\Blambda,j}=0, 
\end{array}\label{M.j.0.2.3.rel}
\end{equation}
Hence, for each $j\geq 2$, $U_1^j$ and $U_2^j/U_1^j$ are highest weight modules generated by highest weight vectors of weights  $3(\lambda_3+\omega_2)+2(\lambda_2-\theta)+j\theta$ and $2\lambda_2+3(\lambda_3+\omega_1)$ respectively and the generators of $U_1^j$ and $U_2^j/U_1^j$ are annihilated by $h\otimes t^r$ for all $h\in \mathfrak h$ and $r\in \mathbb N$. Therefore, for each $j\geq 2$, $U_1^j$ and $U_2^j/U_1^j$ are quotients of the local Weyl modules   $W_{loc}(3(\lambda_3+\omega_2)+2(\lambda_2-\theta)+j\theta)$ and $W_{loc}(2\lambda_2+3(\lambda_3+\omega_1))$ respectively.
 
Using the relations, 
\begin{equation}\begin{array}{l} \label{M.j.0.2.3.rel.b}
y_i(3, (2\lambda_2+3\lambda_3)(h_i)+j-2+p)v_{\Blambda,j}=0, \qquad 
\forall \, i\in I, \\ 
y_i(2, 2(\lambda_2+\lambda_3)(h_i)+j-1+p)v_{\Blambda,j}=0,\qquad \forall \, 
i\in I, \end{array}\end{equation}   
along with \eqref{M.j.0.2.3.rel}, we see that for $j\geq 2$, we have
\begin{equation}\label{U_1.j.1}\begin{array}{l}
y_1(1, (\lambda_2+\lambda_3)(h_1)+j-2)(y_1 \otimes t^{(\lambda_2+\lambda_3)(h_1)+j-1}).v_{\Blambda,j} =0,\\
y_2(1, (\lambda_2+\lambda_3)(h_2)+j)(y_1 \otimes t^{(\lambda_2+\lambda_3)(h_1)+j-1}).v_{\Blambda,j} =0; \\
y_\theta(1, (\lambda_2+\lambda_3)(h_\theta)+j)(y_1 \otimes t^{(\lambda_2+\lambda_3)(h_1)+j-1}).v_{\Blambda,j} =0.\\
y_1(2,2(\lambda_2+\lambda_3)(h_1)+j-3) (y_1 \otimes t^{(\lambda_2+\lambda_3)(h_1)+j-1}).v_{\Blambda,j} =0;\\
y_2(2,2(\lambda_2+\lambda_3)(h_2)+j) (y_1 \otimes t^{(\lambda_2+\lambda_3)(h_1)+j-1}).v_{\Blambda,j} =0.
\end{array}\end{equation} and 
\begin{equation}\label{U_2.j.1}
  \begin{array}{l}y_2(1,(\lambda_2+\lambda_3)(h_2)+j-2)(y_2 \otimes t^{(\lambda_2+\lambda_3)(h_2)+ j-1})v_{\Blambda,j}  = 0,\\
      y_{2}(2, 2(\lambda_2+\lambda_3)(h_2)+j-3)(y_2 \otimes t^{(\lambda_2+\lambda_3)(h_2)+ j-1})v_{\Blambda,j}  = 0,\end{array}\end{equation}
Since $(y_\theta\otimes t^{(\lambda_2+\lambda_3)(h_\theta)+j})v_{\Blambda,j}=0$, for $j\geq 2$, we have 
   \begin{equation}\label{U_2.j.2} (y_{1} \otimes t^{(\lambda_2+\lambda_3)(h_1)+j-1})(y_2 \otimes t^{(\lambda_2+\lambda_3)(h_2)+j-1})v_{\Blambda,j} \in U_1^j.\end{equation}
Further, since \eqref{xrs relation.2} and 
$y_{i}(2,2(\lambda_2+\lambda_3)(h_i)+j-1)v_{\Blambda,j} = 0$ for $i\in I$,  imply 
$${}_{(\lambda_2+\lambda_3)(h_i)+1}y_{i}(2,2(\lambda_2+\lambda_3)(h_i)+j-1)v_{\Blambda,j} = (y_{i} \otimes t^{(\lambda_2+\lambda_3)(h_i)})(y_{i} \otimes t^{(\lambda_2+\lambda_3)(h_i)+j-1})v_{\Blambda,j},$$
we have $y_{\sigma(i)}\otimes t^{(\lambda_2+\lambda_3)(h_{\sigma(i)})+j-1)}y_{i}(2,2(\lambda_2+\lambda_3)(h_i)+j-1)v_{\Blambda,j} = 0$ for $i\in I$, where $\sigma=(12)$ is the non-trivial permutation in the symmetric group $\mathfrak{S}_2$. For $j\geq 2$, this gives,
$$\begin{array}{l}
y_{\sigma(i)}(2,2(\lambda_2+\lambda_3)(h_{\sigma(1)}))+j-1) (y_i\otimes t^{(\lambda_2+\lambda_3)(h_i)+j-1)})v_{\Blambda,j} \\=[ y_i\otimes t^{(\lambda_2+\lambda_3)(h_i)+j-1)} y_{\sigma(i)}\otimes t^{(\lambda_2+\lambda_3)(h_{\sigma(i)})}]y_{\sigma(i)}\otimes t^{(\lambda_2+\lambda_3)(h_{\sigma(i)})+j-1)}v_{\Blambda,j}, 
\end{array} $$
which implies 
\begin{align}\label{U_2.j.3}
&y_{1}(2,2(\lambda_2+\lambda_3)(h_1)+j-1) (y_2\otimes t^{(\lambda_2+\lambda_3)(h_2)+j-1)})v_{\Blambda,j} \in U_1^j,\\
&y_\theta\otimes t^{(\lambda_2+\lambda_3)(h_\theta)+j-1)} y_2\otimes t^{(\lambda_2+\lambda_3)(h_2)+j-1)}v_{\Blambda,j}\in U_1^j.\label{U_2.j.4}
\end{align}
The relations \eqref{U_2.j.1},  \eqref{U_2.j.2},  \eqref{U_2.j.3} and  \eqref{U_2.j.4} along with the fact that $U_2^j/U_1^j$ is a quotient of $W_{loc}(2\lambda_2+3(\lambda_3+\omega_1))$ together imply that there exists a  $\mathfrak{sl}_3[t]$-module surjection $ \tau_{s_{2,j}}M_{j-2}(0, \lambda_2, \lambda_1+\omega_1) \rightarrow U_2^j/U_1^j, $ where $s_{2,j} = (\lambda_2+\lambda_3)(h_2)+j-1).$\\

Now observe that, using relations \eqref{M.rel.3} and \eqref{M.rel.4} in $M_j(0,\lambda_2,\lambda_3)$, from \eqref{xrs relation.2} and $y_{1}(3,3(\lambda_2+\lambda_3)(h_1)+j-1)v_{\Blambda,j}=0$ it follows that,
$$ \dfrac{-1(y_{1}\otimes t^{(\lambda_2+\lambda_3)(h_1)})^{2}}{2}(y_{1}\otimes t^{(\lambda_2+\lambda_3)(h_1)+j-1})v_{\Blambda,j} + {_{(\lambda_2+\lambda_3)(h_1)+1}}y_{1}(3,3(\lambda_2+\lambda_3)(h_1) +j-1)v_{\Blambda,j}=0.$$ Applying $(y_{2} \otimes t^{(\lambda_2+\lambda_3)(h_2)+j-1})^{2}$ on either side of the above equation and 
using\\ $(y_2 \otimes t^{(\lambda_2+\lambda_3)(h_2)+j-1})^{2}v_{\Blambda,j} =0,$ and  
$(y_{\theta} \otimes t^{(\lambda_2+\lambda_3)(h_\theta)+j})v_{\Blambda,j} =0$, we see that
\begin{align*}
&(y_{2} \otimes t^{(\lambda_2+\lambda_3)(h_2)+j-1})^{2}  \big( \dfrac{(y_{1}\otimes t^{(\lambda_2+\lambda_3)(h_1)})^{2}}{2}(y_{1}\otimes t^{(\lambda_2+\lambda_3)(h_1)+j-1})v_{\Blambda,j} \big) \\ &= (y_{2} \otimes t^{(\lambda_2+\lambda_3)(h_2)+j-1})^{2}\big({_{(\lambda_2+\lambda_3)(h_1)+1}}y_{1}(3,3(\lambda_2+\lambda_3)(h_1) +j-1)v_{\Blambda,j}\big) \end{align*}
reduces to 
 \begin{equation}\begin{array}{l}
2(y_{1}\otimes t^{(\lambda_2+\lambda_3)(h_1)})(y_{\theta}\otimes t^{(\lambda_2+\lambda_3)(h_\theta)+j-1})(y_1 \otimes t^{(\lambda_2+\lambda_3)(h_1)+j-1})(y_2 \otimes t^{(\lambda_2+\lambda_3)(h_2)+j-1})v_{\Blambda,j}  \\+ 
 (y_{\theta}\otimes t^{(\lambda_2+\lambda_3)(h_\theta)+j-1})^{2}(y_{1}\otimes t^{(\lambda_2+\lambda_3)(h_1)+j-1})v_{\Blambda,j}  =0.\end{array} \label{U_1.j.3}
 \end{equation}
As $2+2(\lambda_2+\lambda_3)(h_1) +j-1\geq 1+2(\lambda_2+\lambda_3)(h_1)+j$, using relation \eqref{xrs relation.2} of CV modules,  we have ${_{(\lambda_2+\lambda_3)(h_1)}}y_{1}(2,2((\lambda_2+\lambda_3)(h_1)+j-1)v_{\Blambda,j} =0.$ Thus it follows that, 
$$-(y_{1}\otimes t^{(\lambda_2+\lambda_3)(h_1)+j-1})(y_{1}\otimes t^{(\lambda_2+\lambda_3)(h_1)})v_{\Blambda,j} = {}_{(\lambda_2+\lambda_3)(h_1)+1}y_{1}(2,2(\lambda_2+\lambda_3)(h_1) +j-1)v_{\Blambda,j}.$$ Applying  $(y_{2} \otimes t^{(\lambda_2+\lambda_3)(h_2)+j-1})^{2}$ on either side of the above equation and using the same arguments as above we get,
\begin{equation} (y_{1}\otimes t^{(\lambda_2+\lambda_3)(h_1)})(y_{\theta}\otimes t^{(\lambda_2+\lambda_3)(h_\theta)+j-1})(y_1 \otimes t^{(\lambda_2+\lambda_3)(h_1)+j-1})(y_2 \otimes t^{(\lambda_2+\lambda_3)(h_2)+j-1})v_{\Blambda,j} =0.\label{8.25}\end{equation}
From \eqref{8.25} and \eqref{U_1.j.3} it follows that, 
\begin{equation} \label{8.26}
y_\theta(2,2(\lambda_2+\lambda_3)(h_\theta)+2j-2)y_1\otimes t^{(\lambda_2+\lambda_3)(h_1)+j-1}.v_{\Blambda,j} =0.\end{equation} 
The fact that $U_1^j$ is a quotient of $W_{loc}(3(\lambda_3+\omega_2)+2(\lambda_2-\theta)+j\theta)$ along with the relations \eqref{U_1.j.1} and \eqref{8.26}, 
proves that  for $j\geq 2$, there exists a surjective $\mathfrak{sl}_3[t]$-module homomorphism $\tau_{s_{1,j} }M_{j-2}(\omega_2, \lambda_2+\omega_2,\lambda_3) \rightarrow U_1^j,$ where $s_{1,j} = (\lambda_2+\lambda_3)(h_1)+j-1.$\\

Similar arguments show that for $\lambda_1=\omega_1$, there exists a surjective $\mathfrak{sl}_3[t]$-module homomorphism $\phi_j^+ : M_j(\omega_1,\lambda_2,\lambda_3) \rightarrow M_{j-2}(\omega_1,\lambda_2+\theta,\lambda_3)$ such that $\ker \phi_j^+$ has a filtration by modules $U_1^j$ and $U_2^j$ such that 
\begin{align*} U_1^j= \bu(\mathfrak{sl}_3[t])(y_1\otimes t^{s_{3,j}})v_{\Blambda,j}\qquad  
& \qquad U_2^j = \bu(\mathfrak{sl}_3[t])(y_2\otimes t^{s_{4,j}})v_{\Blambda,j}+U_1^j\end{align*} with $s_{3,j} = (\lambda_2+\lambda_3)(h_1)+j$ and  $s_{4,j} = (\lambda_2+\lambda_3)(h_2) +j-1$, and\\ $\tau_{s_{3,j}}:M_{j-1}(0,\lambda_2+\omega_2,\lambda_3)\rightarrow U_1^j$ and $\tau_{s_{4,}} :M_{j-2}(0, \lambda_2 +2\omega_1,\lambda_3)
\rightarrow \dfrac{U_2^j}{U_1^j}$ are  surjective $\mathfrak{sl}_3[t]$-module homomorphisms.\\

As, 
$\dim M_j(\lambda_1,\lambda_2,\lambda_3)= \dim \ker \phi^+_j + \dim M_{j}^+(\lambda_1,\lambda_2, \lambda_3)$, we have, 
\begin{align} 
&\dim M_j(0,\lambda_2,\lambda_3) \leq \dim M_{j-2}(\omega_2,\lambda_2 +\omega_2,\lambda_3) + \dim M_{j-2}(0,\lambda_2,\lambda_3+\omega_1) + \dim M_{j-2}(0, \lambda_2+\theta,\lambda_3),\label{1a} \vspace{.15cm}\\
 &\dim M_j(\omega_1,\lambda_2,\lambda_3) \leq \dim M_{j-1}(0,\lambda_2 +\omega_2,\lambda_3) + \dim M_{j-2}(0,\lambda_2+2\omega_1,\lambda_3) + \dim M_{j-2}(\omega_1, \lambda_2+\theta,\lambda_3).\label{1b}
\end{align}
Substituting the values  of  $\dim M_{0}(\omega_2,\lambda_2 +\omega_2,\lambda_3) , \dim M_{0}(0,\lambda_2,\lambda_3+\omega_1), \dim M_{0}(0, \lambda_2+\theta,\lambda_3)$ in \eqref{1a} and of $\dim M_{1}(0,\lambda_2 +\omega_2,\lambda_3),  \dim M_{0}(0,\lambda_2+2\omega_1,\lambda_3), \dim M_{0}(0, \lambda_2,\lambda_3+\omega_1)$ in \eqref{1b} from \eqref{M_0.dim}, using \lemref{dim V geq}, we see that for
$\lambda_1=\omega_1$ or $0$,
$$\dim M_2(\lambda_1,\lambda_2,\lambda_3) = \dim D(3,3\lambda_3)*D(2,2\lambda_2)*D(1,\lambda_1)*V(\theta)^{*2}.$$ This implies that $\tau_{s_{i,2}}$ is an isomorphism for $i=1,2,3,4$ which proves that the theorem holds for $j=2$. Now applying induction hypothesis assume that the theorem holds for all $0\leq r\leq j$. Then using \eqref{1a}, we see that 
\begin{align*}\dim M_j(0,\lambda_2,\lambda_3)& \leq \dim (D(3, 3\lambda_3) * D(2,2(\lambda_2 + \omega_2)) * D(1,\omega_2)*V(\theta)^{*j-2})\\
&    + \dim( D(3,3(\lambda_3 + \omega_1))* D(2,2\lambda_2)*V(\theta)^{*j-2}) \\
  &   + \dim( D(3,3\lambda_3)*D(2,2(\lambda_2 + \omega_1 + \omega_2))* V(\theta)^{*j-2}) \\
    & = 10^{\lambda_3(h_{\theta})}6^{\lambda_2(h_{\theta}) + 1}8^{j-2}3 + 10^{\lambda_3(h_{\theta}) +1}6^{\lambda_2(h_{\theta})}8^{j-2} + 10^{\lambda_3(h_{\theta})} 6^{\lambda_2(h_{\theta}) +2}8^{j-2}\\
    & = 10^{\lambda_3(h_{\theta})}6^{\lambda_2(h_{\theta})}8^{j}
\end{align*}
which prove $\dim M_j(0, \lambda_2, \lambda_3) \leq \dim(D(3,3\lambda_3)*D(2,2\lambda_2)*V(\theta)^{*j})$. Likewise, using \eqref{1b}, we see that 
$\dim M_j(\omega_1, \lambda_2, \lambda_3) \leq \dim(D(3,3\lambda_3)*D(2,2\lambda_2)*D(1,\omega_1)*V(\theta)^{*j})$. 
Now, the same arguments as those used in \thmref{dim_V_geq.1} show that the above inequalities in the dimensions along with \lemref{dim V geq} imply that the result for all $j\geq 2$. \end{proof}


\subsection{The Module $M_j(\lambda_1,\lambda_2,0)$ for $\lambda_3 =0$} 
If $\lambda_1 =0$, then $M_j(0,\lambda_2,0)\cong_{\mathfrak{sl}_3[t]} D(2,2\lambda_2)*V(\theta)^{*j}$ holds by \thmref{dim_V_geq.1}(i) and \thmref{jgeq2}. Therefore, we can assume that $|\lambda_1|\geq 1$, without loss of generality assume $\lambda_1(h_1)\geq 1$ 
\begin{theorem} \label{Case III} Given $\lambda_1,\lambda_2 \in P^{+}$ with $\lambda_1(h_1)\geq 1$, set  
$$M_j^+(\lambda_1,\lambda_2,0) = \left\{ \begin{array}{ll}
 M_0(\lambda_1 -\omega_1+\omega_2, \lambda_2 +\omega_1,0),  & \text{ when } j=1,\\  M_{j-2}(\lambda_1,\lambda_2 +\theta,0), & \text{ when } j\geq 2. \end{array}\right.$$ 
Then we have the following  exact sequences of $\mathfrak{sl}_3[t]$-modules, 
$$0\rightarrow \ker \phi_j^+ \rightarrow M_{j}(\lambda_1,\lambda_2,0) \xrightarrow{\phi^{+}_j} M_j^+(\lambda_1,\lambda_2,0) \rightarrow 0,$$
where, $\ker \phi_1^+ \cong_{\mathfrak{sl}_3[t]} \tau_{s_1} M_0(\lambda_1 -\omega_1,\lambda_2+\omega_2,0)$ with $s_1 = \lambda_1(h_1)+\lambda_2(h_1)$, and\\
for $j\geq 2$, $\ker \phi_j^+$ has a filtration $0 \subset W^j_{1} \subset W^j_{2} = \ker\phi_j^+$ such that $$W_1^j \cong 
\tau_{s_{1,j}}M_{j-1}(\lambda_1 -\omega_1,\lambda_2+\omega_2,0), \quad 
\text{ and } \quad \frac{W_2^j}{W_1^j} \cong 
\tau_{s_{2,j}}M_{j-2}(\lambda_1-\omega_1,\lambda_2 +2\omega_1,0),$$ with 
$s_{1,j} = \lambda_1(h_1)+\lambda_2(h_1) +j-1$ and 
$s_{2,j} = \lambda_1(h_2)+\lambda_2(h_2) +j-1$. \end{theorem}

\begin{proof} From the definition of $M_1(\lambda_1,\lambda_2,0)$ and $M_{0}(\lambda_1-\omega_1+\omega_2,\lambda_2+\omega_1,0)$, it is clear that there exists 
a surjection $\phi^{+}_1:M_1(\lambda_1,\lambda_2,0) \rightarrow M_{0}(\lambda_1-\omega_1+\omega_2,\lambda_2+\omega_1,0)$ such that 
$$\ker \phi_1^{+} = \bu(\mathfrak{sl}_3[t]).(y_{1} \otimes t^{\lambda_1(h_1)+\lambda_2(h_1)})v_{\Blambda,1} .$$
From the defining relations of $M_1(\lambda_1,\lambda_2)$, it is clear that 
$(y_{1} \otimes t^{\lambda_1(h_1)+\lambda_2(h_1)})v_{\Blambda,1}$ is a highest weight vector of weight $\lambda_1+2\lambda_2+\theta$ and 
$$h\otimes t^r.(y_{1} \otimes t^{\lambda_1(h_1)+\lambda_2(h_1)})v_{\Blambda,1}=0, \forall\, h\in \mathfrak h, \, r\in \mathbb N.$$
Further, as  
$$\begin{array}{ll} y_1(1,\lambda_1(h_1)+\lambda_2(h_1)+1)v_{\Blambda,1} =0,\qquad &y_1(2,\lambda_1(h_1)+2\lambda_2(h_1))v_{\Blambda,1} =0, \\  
y_2(1,\lambda_1(h_2)+\lambda_2(h_2)+1)v_{\Blambda,1}=0, \qquad & y_\theta(1,\lambda_1(h_\theta)+\lambda_2(h_\theta)+1)v_{\Blambda,1}=0, \end{array}$$ it follows that 
\begin{align*}
        (y_1 \otimes t^{\lambda_1(h_1)+\lambda_2(h_1)-1})(y_{1} \otimes t^{\lambda_1(h_1)+\lambda_2(h_1)})v_{\Blambda,1} & =0,\\
        (y_2 \otimes t^{\lambda_1(h_2)+\lambda_2(h_2)+1})(y_{1} \otimes t^{\lambda_1(h_1)+\lambda_2(h_1)})v_{\Blambda,1}& =0.
    \end{align*} Consequently, there exists a surjective map $\phi_1^- : \tau_{s_1}M_0(\lambda_1 - \omega_1,\lambda_2 +\omega_2,0) \longrightarrow \ker \phi^+_1$, implying,  $\dim \ker \phi^+\leq \dim M_0(\lambda_1 - \omega_1,\lambda_2 +\omega_2,0).$ Hence,
\begin{align*}
     \dim M_{1}&(\lambda_1,\lambda_2,0) \\&= \dim \ker\phi^+ + \dim M_0(\lambda_1-\omega_1+\omega_2,\lambda_2+\omega_1,0)\\&
      \leq \dim M_{0}(\lambda_1 - \omega_1,\lambda_2+\omega_2,0) + \dim M_0(\lambda_1-\omega_1+\omega_2, \lambda_2+\omega_1,0)\\& 
      = \dim D(2,2(\lambda_2 + \omega_2))*D(1,\lambda_1 - \omega_1) 
+ \dim D(2,2(\lambda_2 +\omega_1))*D(1,\lambda_1 -\omega_1+\omega_2)\\&
      = 6^{\lambda_2(h_1)+\lambda_2(h_2)+1}3^{\lambda_3(h_1)+\lambda_3(h_2)-1} + 6^{\lambda_2(h_1)+\lambda_2(h_2)+1}3^{\lambda_3(h_1)+\lambda_3(h_2)}\\&
      = 6^{\lambda_2(h_1)+\lambda_2(h_2)}3^{\lambda_3(h_1)+\lambda_3(h_2)}8 = \dim D(2,2\lambda_2)*D(1,\lambda_3)*V(\theta)
 \end{align*}
Coupled with \lemref{dim V geq}, it follows from above that 
\begin{equation} \dim M_{1}(\lambda_1,\lambda_2,0) = \dim D(2,2\lambda_2)*D(1,\lambda_3)*V(\theta),\label{M.1.case3}\end{equation} which implies that the $\mathfrak{sl}_3[t]$-module $\ker\phi_1^+$ is isomorphic to $\tau_{s_1}M_0(\lambda_1-\omega_1,\lambda_2+\omega_2,0)$, with $s_1 = (\lambda_1+\lambda_2)(h_1)$. This proves the theorem for $j=1$.\\ 

\noindent For $j\geq 2$, it is clear from the definitions of $M_j(\lambda_1,\lambda_2,0)$ and $M_{j-2}(\lambda_1,\lambda_2 +\theta,0)$  that there exists a $\mathfrak{sl}_3[t]$-module surjection  $M_j(\lambda_1,\lambda_2,0)\xrightarrow{\phi_j^+} M_{j-2}(\lambda_1,\lambda_2 +\theta,0)$ such that
$$\ker(\phi_j^+) = <y_1 \otimes t^{(\lambda_1+\lambda_2)(h_1)+j-1}.v_{\Blambda,j}, y_2 \otimes t^{(\lambda_1+\lambda_2)(h_2)+j-1})v_{\Blambda,j}>.$$
Setting  \begin{align*} W_1^j = U(\mathfrak{sl}_3[t])(y_1 \otimes t^{(\lambda_1+\lambda_2)(h_1)+j-1})v_{\Blambda,j}, \quad & \quad  W_2^j = U(\mathfrak{sl}_3[t])(y_2 \otimes t^{(\lambda_1+\lambda_2)(h_2)+j-1})v_{\Blambda,2}+ W^j_1 ,\end{align*} it is easy to see that $0\subset W^j_1\subset W^j_2=\ker\phi_j^+$ gives a filtration of $\ker \phi_j^+$. 
 
\noindent From the defining relations of $M_j(\lambda_1,\lambda_2,0)$, it is clear that $(y_1 \otimes t^{(\lambda_1+\lambda_2)(h_1)+j-1})v_{\Blambda,j}$ is a highest weight vector of weight $\lambda_1+2\lambda_2+j\theta-2\omega_1+\omega_2$ that is annihilated by $h\otimes t^r$ for all $h\in \mathfrak h$ and $r\in \mathfrak N$. 
Further, as 
\begin{equation}\begin{array}{ll}y_\alpha(1,(\lambda_1+\lambda_2)(h_\alpha)+j)v_{\Blambda,j}=0, &\forall \, \alpha\in \Phi^+,\\  y_i(2, (\lambda_1+2\lambda_2)(h_i)+j-1)v_{\Blambda,j} =0, \quad &\forall \, i\in I, \end{array}\label{M.j.rel}
\end{equation}
we see that, $(y_1 \otimes t^{\lambda_2(h_1)})(y_1 \otimes t^{(\lambda_1+\lambda_2)(h_1)+j-1})v_{\lambda,j}=0.$ Now using \eqref{M.j.rel}, the condition that $\lambda_1(h_1)\geq 1$ and $j\geq 2$, we get, 
\begin{align*}
         (y_2 \otimes t^{(\lambda_1+\lambda_2)(h_2)+\lambda_1(h_1)+j-1})(y_1 \otimes t^{\lambda_2(h_1)})(y_1 \otimes t^{(\lambda_1+\lambda_2)(h_1)+j-1})v_{\Blambda,j}&=0,\\
     (y_{\theta} \otimes t^{(\lambda_1+\lambda_2)(h_\theta)+j-1})(y_1 \otimes t^{(\lambda_1+\lambda_2)(h_1)+j-1})v_{\Blambda,j} &=0,\\
    (y_1 \otimes t^{(\lambda_1+\lambda_2)(h_1)+j-2})(y_1 \otimes t^{(\lambda_1+\lambda_2)(h_1)+j-1}))v_{\Blambda,j} &=0,\\
    (y_2 \otimes t^{(\lambda_1+\lambda_2)(h_2)+j})(y_1 \otimes t^{(\lambda_1+\lambda_2)(h_1)+j-1})v_{\Blambda,j} &=0.
\end{align*}
Hence there exists a surjective map from $\tau_{s_{1,j}}M_{j-1}(\lambda_1-\omega_1,\lambda_2 +\omega_2,0)$ to $W_1^j$ for $j\geq 2$. \\

\noindent Similar arguments show that
\begin{align*}
     (y_{\theta} \otimes t^{(\lambda_1+\lambda_2)(h_\theta)+j-1})(y_{2}\otimes t^{(\lambda_1+\lambda_2)(h_2)+j-1})v_{\Blambda,j} & =0,\\
    (y_2 \otimes t^{(\lambda_1+\lambda_2)(h_2)+j-2})(y_2 \otimes t^{(\lambda_1+\lambda_2)(h_2)+j-1})v_{\Blambda,j}&=0.
\end{align*} Now, using the condition that $j\geq 2$ and \eqref{M.j.rel} we get, $$(y_1 \otimes t^{(\lambda_1+\lambda_2)(h_1)+j-1})(y_2 \otimes t^{(\lambda_1+\lambda_2)(h_2)+j-1})v_{\Blambda} \in W_{1}^j,$$ which implies that $W_2^j/W_1^j$ is a quotient 
of $M_{j-2}(\lambda_1-\omega_1,\lambda_2+2\omega_1,0)$. 
Hence, 
\begin{align*}
    &\dim M_j(\lambda_1,\lambda_2,0) = \dim \ker\phi_j^+ + \dim M_{j-2}(\lambda_1,\lambda_2+\theta,0) \\
    & \leq \dim M_{j-1}(\lambda_1-\omega_2, \lambda_2 +\omega_2,0) + \dim M_{j-2}(\lambda_1-\omega_1, \lambda_2 +2\omega_1,0) + \dim M_{j-2}(\lambda_1,\lambda_2+\theta,0).\end{align*} 
Substituting the values of
$\dim M_{1}(\lambda_1 -\omega_1,\lambda_2 +\omega_2,0)$ , $\dim M_{0}(\lambda_1-\omega_1,\lambda_2 +2\omega_1,0)$ and $\dim M_{0}(\lambda_1,\lambda_2 +\theta,0)$ 
from \eqref{M.1.case3} and \eqref{M_0.dim},  we see that 
$\dim M_2(\lambda_1,\lambda_2,0)= \dim D(2, 2\lambda_2)\ast D(1,\lambda_1)\ast V(\theta)^{\ast 2}$. This implies 
that 
$$W_2^2/W_1^2\cong_{\mathfrak{sl}_3[t]} \tau_{s_{2,2}}M_{0}(\lambda_1-\omega_1,\lambda_2+2\omega_1,0), \qquad W_1^2 \cong_{\mathfrak{sl}_3[t]} \tau_{s_{1,2}}M_{1}(\lambda_1-\omega_1,\lambda_2 +\omega_2,0). $$ Thus the theorem holds for $j=2$. Now assuming 
that the theorem holds for all positive integers $2\leq p<j$. Then using induction hypothesis it follows from the above computations that
\begin{align*}
&\dim M_j(\lambda_1,\lambda_2,0)\\
    & \leq  \dim(D(2,(\lambda_2 +2\omega_2))*D(1,\lambda_3-\omega_1)*V(\theta)^{*j-1})\\
    &+\dim(D(2,2(\lambda_2 +2\omega_1))*D(1,\lambda_3-\omega_1)*V(\theta)^{*j-2})+ \dim (D(2,2(\lambda_2 +\theta))*D(1,\lambda_3)*V(\theta)^{*j-2})\\
    & = 6^{\lambda_2(h_{\theta})+1}3^{\lambda_1(h_{\theta})-1}8^{j-1}+6^{\lambda_2(h_{\theta})+2}3^{\lambda_1(h_{\theta})-1}8^{j-2}+6^{\lambda_2(h_{\theta})+2}3^{\lambda_1(h_{\theta})}8^{j-2}\\
    & = 6^{\lambda_2(h_{\theta})}3^{\lambda_1(h_{\theta})}8^{j} = \dim (D(2,2\lambda_2)*D(1,\lambda_3)*V(\theta)^{*j}).
\end{align*}
Thus using \lemref{dim V geq} and the same arguments as used for $j=1$ case,  
we see that the theorem holds for all $j\in \mathbb N$. \end{proof}

\subsection{Proof of \thmref{Main.thm.8}(ii)} 

First consider the case when $j=1$. 
As noted in \secref{Mjmod}, $M_1(0,\lambda_2,0)$ and $M_1(0,0,\lambda_3)$  are isomorphic to the Demazure modules $D(2,2\lambda_2+\theta)$ and $D(3,3\lambda_3+\theta)$ respectively. Since $\theta(h_\theta)=2$ and $\theta(h_i)=1$, using the prime factorization of Demazure modules of $\mathfrak{g}[t]$ for $\mathfrak g$ of type $A$ and $D$ \cite[Proposition 3.9]{MR3478864}, we see that
$D(\ell,\ell\lambda_\ell +\theta)\cong_{\mathfrak{sl}_3[t]}D(\ell, \ell\lambda_\ell)\ast V(\theta)$ for $\ell\geq 2.$ Hence the result holds for  
$M_1(0,\lambda_2,0)$ and $M_1(0,0,\lambda_3)$. 

By definition, $M_1(\lambda_1,0,0)$ is isomorphic to 
$M(\lambda_1,\theta)$ as a $\mathfrak{sl}_3[t]$-module. Therefore by \propref{M.mu.nu}(i), $M_1(\lambda_1,0,0)\cong_{\mathfrak{sl}_3[t]}D(1, \lambda_1)\ast V(\theta)$

\noindent Now note that when at least two of the dominant integral weights $\lambda_1, \lambda_2,\lambda_3$ are non-zero,
$$M_1(\lambda_1,\lambda_2,\lambda_3)\cong_{\mathfrak{sl}_3[t]} D(3, 3\lambda_3)\ast D(2, 2\lambda_2)\ast D(1, \lambda_1)\ast V(\theta),$$ follows
from \thmref{dim_V_geq.1} and \corref{dim_V_geq.1} when $|\lambda_1|\leq 1$ and 
from \thmref{Case III} when $\lambda_3=0$. This proves \thmref{Main.thm.8}(ii) for $j=1$.

For $j\geq 2$, we see that \thmref{Main.thm.8}(ii) clearly follows from 
\lemref{dim V geq} and \thmref{Case II} when $|\lambda_1|\leq 1$ and from \lemref{dim V geq} and \thmref{Case III} when $\lambda_3=0$. \qed

\subsection{ Proof of \thmref{trunc}(i)}  As observed in \remref{rem.trunc}, 
the proof of \thmref{trunc}(i) follows from \thmref{Main.thm.8} for $\lambda_2=0$ and $\lambda_1=\lambda-j\theta$, $0\leq j\leq L_\lambda$  .\qed
\\

{\it{Acknowledgements.}} Part of this work was completed while the first two authors attended the workshop, "Representation Theory of Real Lie Groups and Automorphic Forms" at the Harish-Chandra Research Institute (HRI). The authors are grateful to the HRI for providing an excellent working environment.
\\

{\it{Funding.}} The first author is grateful to the Indian Institute of Science Education and Research Mohali for PhD fellowship and support.


 
\end{document}